\documentclass[10pt]{article}
\usepackage{amsfonts} 
\usepackage{amsmath,hhline,latexsym,mathrsfs,amsthm} 
\usepackage{amssymb}
\usepackage{relsize}
\DeclareMathAlphabet\mathbfcal{OMS}{cmsy}{b}{n}
\usepackage{mathtools}
\usepackage{dsfont}

\usepackage{hyperref,enumitem}
\usepackage{cleveref}

\usepackage{xcolor}
\usepackage{color}
\usepackage{graphicx}
\usepackage{comment}
\usepackage{subfigure}

\renewcommand{\leq}{\leqslant}
\renewcommand{\le}{\leqslant}
\renewcommand{\geq}{\geqslant}
\renewcommand{\ge}{\geqslant}

\usepackage[english]{babel} 

\usepackage[colored]{shadethm} 

\definecolor{shadethmcolor}{rgb}{1,0.871,0.790}

\definecolor{shaderulecolor}{rgb}{0.651,0.074,0.090}
\newshadetheorem{theorem}{Theorem} 
\newshadetheorem{prop}{Proposition} 
\newshadetheorem{lemma}{Lemma} 
\newshadetheorem{remark}{Remark} 
\newshadetheorem{defi}{Definition}
\newshadetheorem{cor}{Corollary}

\newcommand{\R}{\mathbb{R}}
\newcommand{\N}{\mathbb{N}}

\usepackage{ifdraft}
\ifoptionfinal{
\usepackage[disable]{todonotes}
}{
\usepackage[norefs, nocites]{refcheck}
\usepackage[breakall, fit]{truncate}
\usepackage{accsupp}
\setlength{\marginparsep}{5pt}

\usepackage[bordercolor=white, color=white, textwidth=100pt]{todonotes}
}

\makeatletter\providecommand\@dotsep{5}\def\listtodoname{List of Todos}\def\listoftodos{\hypersetup{linkcolor=black}\@starttoc{tdo}\listtodoname\hypersetup{linkcolor=blue}}\makeatother

\newcommand{\verti}[1]{{\left\vert\kern-0.25ex\left\vert\kern-0.25ex\left\vert #1 
    \right\vert\kern-0.25ex\right\vert\kern-0.25ex\right\vert}}

\newcommand\C{\mathcal{C}}
\newcommand\CC{\mathcal{C}}

\newcommand\yy{\widehat{y}}

\newcommand\ytt{y_{tt}}
\newcommand\yxx{\Delta y}

\newcommand\intq{\int_Q}

\def\dx{\,\textnormal{d}x}
\def\dt{\textnormal{d}t}
\def\d{\,\textnormal{d}}

\title{On the exact boundary controllability of semilinear wave equations}

\author{
 Sue Claret \thanks{Universit\'e Clermont Auvergne, CNRS, LMBP, F-63000 Clermont-Ferrand, France; sue.claret@uca.fr.}
	\and J\'er\^ome Lemoine \thanks{Universit\'e Clermont Auvergne, CNRS, LMBP, F-63000, Clermont-Ferrand, France; jerome.lemoine@uca.fr.}
\and Arnaud M\"unch \thanks{Universit\'e Clermont Auvergne, CNRS, LMBP, F-63000 Clermont-Ferrand, France;  arnaud.munch@uca.fr. Corresponding author.}
}

\begin{document}
\maketitle

\begin{abstract}
We address the exact boundary controllability of the semilinear wave equation $\ytt-\yxx + f(y)=0$ posed over a bounded domain $\Omega$ of $\mathbb{R}^d$.
Assuming that $f$ is continuous and satisfies the condition $\limsup_{\vert r\vert\to \infty} \vert f(r)\vert /(\vert r\vert \ln^p\vert 
r\vert)\leq \beta$ for some $\beta$ small enough and some $p\in [0,3/2)$, we apply the Schauder fixed point theorem to prove the uniform controllability for initial data in $L^2(\Omega)\times H^{-1}(\Omega)$. 
Then, assuming that $f$ is in $\C^1(\R)$  and satisfies the condition $\limsup_{\vert r\vert\to \infty} \vert f^\prime(r)\vert/\ln^p\vert 
r\vert\leq \beta$, we apply the Banach fixed point theorem and exhibit a strongly convergent sequence to a state-control pair for the semilinear equation. 
\end{abstract}
		
\textbf{AMS Classifications:} 35L71, 93B05.

\textbf{Keywords:} Semilinear wave equation, Exact boundary controllability, Carleman estimates, Fixed point.

\section{Introduction and main results}
	

Let $\Omega$ be a bounded domain in $\mathbb{R}^d$ of class $\CC^2$ and let $T>0$. We set $Q:=\Omega\times (0,T)$ and $\Sigma:=\partial\Omega\times (0,T)$. We consider the semilinear problem in $y=y(x,t)$
\begin{equation}\label{main_control_problem}
Ly + f(y) =0 \text{ in } Q, \qquad y =v 1_{\Gamma_0}\text{ on } \Sigma, \qquad (y(\cdot,0), y_t(\cdot,0))= (u_0,u_1) \text{ in } \Omega,
\end{equation}
where  $L:=\partial^2_{t}-\Delta$ denotes the wave operator,  $(u_0, u_1) \in \boldsymbol{H}:=L^2(\Omega) \times H^{-1}(\Omega)$ is a given initial state, $v\in L^2(\Sigma)$ is a control function and $f$ a continuous function over $\R$. $\Gamma_0$ denotes a non empty subset of $\partial\Omega$. 
	
The exact boundary controllability problem  associated to \eqref{main_control_problem} states as follows: 
	 given $T>0$, $\Gamma_0\subset \partial\Omega$ and $(u_0, u_1), (z_0,z_1)\in \boldsymbol{H}$, find a control function $v\in L^2(\Sigma)$ and $y\in \C^0([0,T];L^2(\Omega))\cap \C^1([0,T]; H^{-1}(\Omega))$ solution of \eqref{main_control_problem} and such that $(y(\cdot,T), y_t(\cdot,T))=(z_0,z_1)$ in $\Omega$.

The linear problem \eqref{main_control_problem} with $f\equiv 0$ is exactly controllable provided that $T>0$ and $\Gamma_0$ are sufficiently large (see \cite[Theorem 6.1, p. 60]{Lions-1} and \cite[Theorem 4.9, p. 1058]{Bardos-Lebeau-Rauch}). In the nonlinear case, a first exact boundary controllability result has been given in \cite[Theorem 2.1]{Zuazua-2} assuming $f$ globally Lipschitz and initial data in $H^{\gamma}_0(\Omega) \times H^{\gamma-1}(\Omega)$ for $\gamma \in (0,1), \gamma\ne \frac{1}{2}$ leading to Dirichlet control in $H^{\gamma}_0(0,T; L^2(\Gamma_0))$. A Schauder fixed point argument is used coupled with the HUM method developed in \cite{Lions-1}. Still assuming $f^\prime\in L^\infty(\R)$, \cite[Theorem 1.1]{Lasi-Trig} covers the case $\gamma=0$ and generalizes the result to semilinear abstract systems by using a global inversion theorem. 
We also mention~\cite{coron-trelat-wave-06} where a boundary controllability result is proved in the one-dimensional case for a specific class of initial and final data and~$T$ large enough by a quasi-static deformation approach.

Assuming $(u_0,u_1)\in \boldsymbol{V}:=H_0^1(\Omega)\times L^2(\Omega)$, the boundary controllability may also be obtained indirectly with the domain extension from interior controllability results. In this respect, 
we mention \cite[Theorem 1]{Zuazua-ihp} assuming $\Omega=(0,1)$, $T>2$ and that $f\in \C^1(\R)$ does not grow faster at infinity than $\beta \vert r\vert  \ln^2(\vert r\vert)$ for some $\beta>0$ small enough, for a global controllability result in $\boldsymbol{V}$.  This result has been extended to any spatial dimension, first in \cite[Theorem 3.1]{Li_Zhang_2000} and then in \cite[Theorem 4.5, page 116]{Zhang_springerbriefs}
assuming that $f\in \C^1(\R)$ does not grow faster at infinity than $\beta \vert r\vert  \ln^p(\vert r\vert)$ for some $\beta>0$ small enough for $p=1/2$ and any $0<p<3/2$ respectively. The above results are based on the Schauder theorem together with an estimate of the cost of control for linear wave equations with potential derived using Carleman estimates (we refer to \cite[Theorem 2.2, page 8]{Duyckaerts_2008}). Eventually, we mention 
~\cite{DehmanLebeau} dealing with subcritical nonlinearities satisfying the sign condition $rf(r)\geq 0$ for all $r\in \mathbb{R}$ (weakened later in~\cite{JolyLaurent2014} to an asymptotic sign condition leading to a semi-global controllability result, in the sense that the final data $(z_0,z_1)$ must be prescribed in a precise subset of  $\boldsymbol{V}$). 

In this work, we directly address the exact boundary controllability for \eqref{main_control_problem} under the usual conditions on $(u_0,u_1), T$ and $\Gamma_0$ encountered in the linear case but with respect to \cite{Lasi-Trig}, by replacing the condition $f^\prime\in L^\infty(\R)$ by the slightly super-linear condition used in \cite{Zhang_springerbriefs}. Our result is as follows.


%
\begin{theorem}\label{main_th} 
  For any $x_0\in \mathbb{R}^d\backslash\overline{\Omega}$, let $\Gamma_1 := \{x\in \partial\Omega : (x-x_0)\cdot \nu(x)>0\}$ and $\Gamma_0\subset \partial\Omega$ such that  $ \textrm{dist}(\Gamma_1,\partial\Omega\setminus\Gamma_0)>0$ and let  $T>2\max_{x\in\overline{\Omega}}\vert x-x_0\vert$.
\begin{itemize}
\item Assume that $f\in \CC^0(\mathbb{R})$ and that there exists $0\le p<3/2$ such that $f$  satisfies
\begin{enumerate}[label=$\bf (H_p)$,leftmargin=1.5cm]
\item \label{H1}$\exists \alpha_1,\alpha_2, \beta^\star >0, \hspace{0.2cm} |f(r)| \leq \alpha_1 + |r|\left( \alpha_2 + \beta^\star\ln_+^p(r) \right), \hspace{0.25cm} \forall r\in \mathbb{R}.$
\end{enumerate}
If $\beta^\star  $ is small enough,  then for any initial state $(u_0, u_1)$ and final state $(z_0,z_1)$ in $\boldsymbol{H}$,
 system \eqref{main_control_problem} is exactly controllable with controls in $L^2(\Sigma)$.
 \item Assume that $f\in \CC^1(\mathbb{R})$ and that there exists $0\le p<3/2$ such that $f$  satisfies
\begin{enumerate}[label=$\bf (H'_{p})$,leftmargin=1.5cm]
\item \label{asymptotic_behavior_prime_p}   $ \exists \alpha, \beta^\star>0, \hspace{0.2cm} |f^\prime(r)|\leq \alpha+ \beta^\star \ln_+^{p}(r), \hspace{0.25cm} \forall r\in \mathbb{R}.$
\end{enumerate}
 If $\beta^\star  $ is small enough, then for any initial state $(u_0, u_1)$ and final state $(z_0,z_1)$ in $\boldsymbol{H}$, one can construct a non trivial sequence $(y_k, v_k)_{k\in \mathbb N^\star}$ that converges strongly to a controlled pair $(y,v)$ in $\big( \C^0([0,T];L^2(\Omega))\cap \C^1([0,T];H^{-1}(\Omega)) \big)\times L^2(0,T; L^2(\partial\Omega))$ for system \eqref{main_control_problem}. Moreover, the convergence of $(y_k,v_k)_{k\in \N^\star}$ holds at least with a linear rate for the norm  $\Vert \rho\cdot \Vert_{L^2(Q)}+\Vert \rho\cdot \Vert_{L^2(\Sigma)}$ where $\rho=\rho(x,t,s)$ is defined in \eqref{weights} and $s$ is chosen sufficiently large depending on $\Vert (u_0,u_1)\Vert_{\boldsymbol{H}}$
and $\Vert (z_0,z_1)\Vert_{\boldsymbol{H}}$.

 \item Assume that $f\in \CC^0(\mathbb{R})$ and satisfies $\bf (H_{3/2})$, i.e. \ref{H1} with $p=3/2$. If $\beta^\star  $ is small enough,  then for any initial state $(u_0, u_1)$ and final state $(z_0,z_1)$ in $\boldsymbol{V}$, system \eqref{main_control_problem} is exactly controllable with controls in $H^1(0,T;L^2(\partial\Omega))\cap \C^0([0,T];H^{1/2}(\partial\Omega))$.
 \end{itemize}
\end{theorem}
\Cref{main_th} extends and generalizes to any dimension \cite{Bhandari_MCSS_2022} devoted to $d=1$. It relies on fixed point approaches in a functional class $\C(s)$
dependent of a Carleman parameter $s$ large enough. For any $\widehat y\in \C(s)\subset L^\infty(0,T; L^2(\Omega))$, the fixed point operator $\Lambda_s$ is such that $y:=\Lambda_s(\widehat y)\in \C(s)$ is a controlled solution of 
\begin{equation}\label{linear_control_problem}
Ly= B  \text{ in } Q, \qquad y =v\, 1_{\Gamma_0} \text{ on } \Sigma, \qquad (y(\cdot,0), y_t(\cdot,0))= (u_0,u_1)  \text{ in } \Omega,
\end{equation} 
%
%
with $B= -f(\widehat y)$ satisfying $(y(\cdot, T),y_t(\cdot,T)) = (z_0,z_1)$ through the boundary function $v$; the pair $(y,v)$ is chosen as the minimizer of a quadratic functional (see \eqref{extremal-2}) involving Carleman weights and cut-off time and space functions. While $\C(s)$ is a subset of $L^\infty(Q)$ in \cite{Bhandari_MCSS_2022}, the class here is a subset of $L^\infty(0,T; L^2(\Omega))$ allowing to consider any spatial dimension with a simpler proof. 

We emphasize that we get the exact controllability for \eqref{main_control_problem} under the conditions encountered in the linear situation: the controllability time $T>T(x_0):=2\max_{x\in\overline{\Omega}}\vert x-x_0\vert$ and support $\Gamma_0$ satisfy the usual geometric conditions introduced in \cite{Lions-1} while the initial data is assumed in $\boldsymbol{H}$.
This is in contrast with \cite[Theorem 4.5]{Zhang_springerbriefs} devoted to distributed controllability where the initial data is taken in $\boldsymbol{V}$ with a controllability time greater
than $\max(2T(x_0), c \,d^{3/2})$ for some $c>0$ (we refer to \cite[Remark 4.7 page 118]{Zhang_springerbriefs}). Moreover, we consider nonlinearities $f$ in $\C^0(\R)$ while \cite[Theorem 4.5]{Zhang_springerbriefs} assumes $f$ in $\C^1(\R)$; this is due to the fact that the linearization \eqref{linear_control_problem}, where the nonlinear term is seen as a right hand side, does not involve any derivative of $f$. As a matter of fact, the first item in \Cref{main_th} can not be obtained from the extension domain method and controllability results for the distributed case.  

Remark also that the third item includes the value $p=3/2$ in the exponent of the logarithm term (see \ref{H1}) contrary to \cite[Theorem 4.5]{Zhang_springerbriefs}. Last, the second item, with a growth condition on the derivative of $f$, provides a constructive way to approximate control-state pair for \eqref{main_control_problem}, which is fundamental for applications. The regularity is used to estimate some $L^2(L^q)$ norm of $f(y_1)-f(y_2)$ for any elements $y_1,y_2$ in $\C(s)$ (see \Cref{Prop-Banach_contraction}). It relaxes the H\"older assumption on $f^\prime$ used in \cite[Theorem 2.3]{Munch-Trelat} 
based on a Newton type linearization. 
To our knowledge, this is the first result leading to a convergent approximation of boundary controls for superlinear nonlinearities without smallness assumption notably on the initial condition and target (in contrast to the recent works \cite{Rosier2021, Weiss_semilinear}).


As in \cite{Bhandari_MCSS_2022}, the crucial technical point in the analysis is a regularity property of the state-control trajectories $(y,v)$ for 
\eqref{linear_control_problem}. We show and use that if the initial condition belongs to $\boldsymbol{V}$ and if the right hand side $B$ belongs to $L^2(Q)$, then the controlled trajectory $y$ solution of \eqref{linear_control_problem} so that $(y,v)$ is the minimizer of the quadratic functional $J_s$ (see Remark \ref{remark_J_s}) belongs to $\C^0([0,T], H^1(\Omega))$.

\paragraph{Outline -} Section \ref{sec:wave_linearized} discusses the exact null controllability of \eqref{linear_control_problem} and provides precise estimates of the control-state pair $(y,v)$ in term of the Carleman parameter $s$ according to the regularity of the data $(u_0,u_1)$ and of the right hand side $B$: \Cref{Lemma_relation_control} for $(u_0,u_1,B)\in \boldsymbol{H}\times L^2(0,T;H^{-r}(\Omega))$, $r\in [0,1]$, $r\not=\frac12$, and \Cref{Thm-regularity-one} for $(u_0,u_1,B)\in \boldsymbol{V}\times L^2(Q)$. Section \ref{Controllability_results} is devoted to the proof of \Cref{main_th}:  in Section \ref{sec:FixedPointPart1}, we employ the Schauder fixed point theorem to the operator $\Lambda_s$ (see \eqref{operator_fixed_point}) defined on the class $\C(s)\subset L^\infty(0,T;L^2(\Omega))$ and prove the first item. In Section \ref{sec:FixedPointPart2}, we prove that, if the nonlinearity $f$ satisfies \ref{asymptotic_behavior_prime_p}, then the operator $\Lambda(s)$ is contracting leading to the second item. Eventually, in Section \ref{sec:FixedPointPart3} assuming the initial data in $\boldsymbol{V}$, we define the operator $\Lambda(s)$ on a class $\widetilde\C(s)$ subset of $H^1(Q)$ and reach the limit case $p=3/2$ in the logarithmic exponent, as announced in the third item. 

In the sequel, $C$ denotes a generic constant which may changes from line to line, but depends only on $\Omega$ and $T$.

\section{Controllability results for the linear wave equation}\label{sec:wave_linearized}

Existence of $L^2(\Sigma)$ controls for \eqref{linear_control_problem} with initial data in $\boldsymbol{H}$ and right hand side in $L^2(0,T;H^{-1}(\Omega))$ is well-known (we refer to \cite[chapter 2]{Lions-1}); corresponding controlled solution belongs to $\C^0([0,T];L^2(\Omega))\cap \C^1([0,T];H^{-1}(\Omega))$. 
For any initial data  $(u_0,u_1)\in \boldsymbol{V}$, right hand side $B$ in $L^2(Q)$ and $T>0$ large enough, we analyze the existence of a control function $v\in H^1(0,T;L^2(\partial\Omega))\cap \C^0([0,T];H^{1/2}(\partial\Omega))$ such that the solution $y$ of \eqref{linear_control_problem} 
		satisfies $(y(\cdot,T), y_t(\cdot,T))=(z_0,z_1)$ and is bounded in $L^\infty(0,T; L^2(\Omega))$. Moreover, we aim to get precise weighted estimates of a particular state-control pair  in term of the data, which will be crucial to handle the nonlinear system \eqref{main_control_problem}. As in \cite{Bhandari_MCSS_2022}, we employ a global Carleman estimates from \cite{Ervedoza-Carleman} as a fundamental tool.

We introduce the usual geometric condition (see \cite[Condition (1.2)]{Ervedoza-Carleman}): for any $x_0\notin \overline{\Omega}$, we introduce $\Gamma_1, \Gamma_0\subset\partial\Omega$ such that
\begin{equation}\label{Geometric_Condition}
\Gamma_1 := \{x\in \partial\Omega : (x-x_0)\cdot \nu(x)>0\}, \qquad \textrm{dist}(\Gamma_1,\partial\Omega\setminus\Gamma_0)>0.
\end{equation}
Let $\Psi\in \CC^2(\partial\Omega)$ be a cut-off  function such that 
\begin{equation}\label{function_Psi}
			0\le \Psi\le 1  \quad \text{on  } \partial\Omega, \qquad
			\Psi=1  \quad \text{on  } \Gamma_1,\qquad
			\Psi=0 \quad\text{on } \partial\Omega\backslash \Gamma_0.
		\end{equation}
We assume that 
\begin{equation}\label{condT}
T> 2\max_{\overline{\Omega}} |x-x_0|
\end{equation}		
and define, for any $\delta >0$ such that $T-2\delta > 2\max_{\overline{\Omega}} |x-x_0|$, a cut-off function $\eta\in\mathcal{C}_c^1(\mathbb{R})$ satisfying 
	\begin{align}\label{function_eta}
			0\leq \eta(t) \leq 1 \text{ in } (\delta,T-\delta), \qquad
			\eta(t) =0 \text{ in } (-\infty,\delta] \cup [T-\delta, +\infty).	
	\end{align} 
Then, for any $\beta\in (0,1)$ and $\lambda>0$, we define the functions 
$\psi(x,t) = |x-x_0|^2 -\beta \big(t-\frac{T}{2}\big)^2 + M_0$, $\phi(x,t)=e^{\lambda \psi(x,t)}$ in $Q$, 
	with $M_0>0$ large enough so that $\psi>0$ in $\overline{Q}$. Then, for all $s\geq s_0$, we define the weight function
	\begin{equation}\label{weights}
		\rho(x,t):= e^{-s\phi(x,t)}  \quad \forall (x,t) \in Q.
	\end{equation}
Remark that $e^{-cs}\leq \rho\leq e^{-s}$ in $Q$
 with $c:=\Vert \phi\Vert_{L^\infty(Q)}$ and $\rho,\rho^{-1}\in \C^\infty(\overline{Q})$.  \\
Let then $P:=\{w\in \C^0([0,T];H_0^1(\Omega))\cap \C^1([0,T];L^2(\Omega)), Lw \in L^2(Q)\}$
and recall that $\partial_\nu w \in L^2(\Sigma)$ for every $w\in P$ (see \cite[Theorem 4.1]{Lions-1}). The global Carleman estimate mentioned earlier reads as follows.

\begin{prop}\label{Thm-Carleman} 
	     For any $x_0\notin\overline{\Omega}$, we assume \eqref{Geometric_Condition} and \eqref{condT}. There exists $s_0>0$, $\lambda>0$ and $C>0$, such that for any $s\geq s_0$ and every $w\in P$	    	%
	    	\begin{multline}\label{Carleman-in-0-T}
	    		s \intq \rho^{-2} (| w_{t}|^2 + | \nabla w |^2)\dx\dt  + s^3 \intq  \rho^{-2} | w|^2  \dx\dt \\
	    		+	s \int_\Omega \rho^{-2}(0) (|w_t(\cdot,0)|^2 + | \nabla w (\cdot,0)|^2)\dx + s^3 \int_\Omega  \rho^{-2}(0) |w(\cdot,0)|^2\dx   \\
	    		\leq C\biggl(\intq \rho^{-2} |Lw|^2\dx\dt   + s \int_{\Sigma} \eta^2(t)\Psi(x)\rho^{-2} | \partial_\nu w|^2\dx\dt\biggr).  
	    	\end{multline}
	\end{prop}
	\begin{proof}
	We refer to \cite[Theorem 2.5 and Remark 2.9]{Ervedoza-Carleman}. 
	\end{proof}
	
\subsection{Estimates for the state-control pair in $\C^0([0,T];L^2(\Omega))\times L^2(\Sigma)$} \label{Section-regularity}

In all the sequel, we suppose that $s_0\ge 1$.
\Cref{Thm-Carleman} allows to deduce the controllability for \eqref{linear_control_problem}
with estimates of the state-control pair in $L^2(Q)\times L^2(\Sigma)$. For any $s\geq s_0$, we define the bilinear form 
\begin{align}\label{form_bilinear}
(w,z)_{P,s} := \intq \rho^{-2} Lw Lz \dx\dt  + s  \int_{\Sigma}  \eta^2(t)\Psi(x)\rho^{-2} \partial_\nu w  \partial_\nu z\dx\dt, \quad \forall w,z\in P.
\end{align}
According to \eqref{Carleman-in-0-T}, \eqref{form_bilinear} defines a scalar product in $P$ and if $P_s$ denotes $P$ endowed with this scalar product, then $P_s$ is an Hilbert space. 	
	We now state the controllability result for the system \eqref{linear_control_problem} (without loss of generality in the null controllability case, for which $(z_0,z_1)=(0,0)$ in $\Omega$).

	\begin{prop}\label{Lemma_relation_control}
	  For any $x_0\notin\overline{\Omega}$, we assume \eqref{Geometric_Condition} and \eqref{condT}. Let $\eta\in \C_c^1(\mathbb{R})$ and $\Psi\in \mathcal{C}^2(\partial\Omega)$ be cut-off functions satisfying  \eqref{function_Psi} and \eqref{function_eta} respectively.
For $s\geq s_0 $, $B \in L^2(0,T, H^{-r}(\Omega))$, $r\in [0,1]$, $r\neq 1/2$ and  $(u_0, u_1) \in \boldsymbol{H}$, there exists a unique $w \in P_s$ such that, for all $z\in P_s$,  
		\begin{align}\label{bilinear equation}
			(w, z)_{P,s}=  <u_1,\,z(\cdot,0)>_{H^{-1}(\Omega),H_0^1(\Omega)}  - \int_\Omega u_0\, z_t(\cdot,0)\dx  + <B, z>_{L^2(0,T; H^{-1}(\Omega)), L^2(0,T;H_0^1(\Omega))}.
		\end{align} 
	 Then $v:= s\eta^2\Psi\rho^{-2}\partial_\nu w $ is a control function for \eqref{linear_control_problem} and $y:=\rho^{-2}L w$ is the associated controlled trajectory, that is $y(\cdot,T)=y_t(\cdot,T)=0$ in $\Omega$. 
	 
		Moreover, there exists a constant $C_r>0$ independent of $s$ such that
		\begin{equation}\label{weighted_estimate_control_traj-frac}
		\begin{aligned}
			\|\rho &y\|_{L^2(Q)} + s^{-1/2}\left\|\frac{\rho}{\eta\Psi^{1/2}} v\right\|_{L^2(\Sigma)}+s^{-2}\|\rho y\|_{L^\infty(0,T;L^2(\Omega))}+s^{-2}\|(\rho y)_t\|_{L^\infty(0,T;H^{-1}(\Omega))}\\
			&\leq C_r \left(s^{r-3/2} \|\rho B\|_{L^2(0,T;H^{-r}(\Omega))} + s^{-1/2}  \|\rho(0)  u_0\|_{L^2(\Omega)} + s^{-1/2}   \|\rho(0) u_1\|_{H^{-1}(\Omega)}  \right).
		\end{aligned}
		\end{equation}
\end{prop}
	
	\begin{remark}
	With no vanishing target $(z_0,z_1)\in \boldsymbol{H}$, the right hand side of \eqref{weighted_estimate_control_traj-frac}
	contains the term $s^{-1/2}   \|\rho(0)  z_0\|_{L^2(\Omega)} + s^{-1/2}   \|\rho(0) z_1\|_{H^{-1}(\Omega)}$. We choose in the sequel  $(z_0,z_1)=(0,0)$ which will make the computations shorter.
\end{remark}
	
\begin{proof} 
Some parts of the proof are only sketched as they are very similar to \cite[Theorem 6]{Bhandari_MCSS_2022} devoted to the case $r=0$ (and $d=1$).
For any $0\le r\le 1$, $r\neq 1/2$ and $z\in P_s$, we have from \eqref{Carleman-in-0-T} :
			$$
			\begin{aligned}
				&\left| <B, z>_{L^2(0,T; H^{-1}(\Omega)),  L^2(0,T;H_0^1(\Omega))} \right|=\left| <B, z>_{L^2(0,T; H^{-r}(\Omega)),  L^2(0,T;H_0^r(\Omega))} \right|  \\
				&\hspace{1cm}\leq\|\rho B\|_{L^2(0,T;H^{-r}(\Omega))}  \|\rho^{-1}z\|_{L^2(0,T;H^{r}(\Omega))}\\
				&\hspace{1cm}\leq C\|\rho B\|_{L^2(0,T;H^{-r}(\Omega))}  \|\rho^{-1} z\|^{1-r}_{L^2(Q)}\|\nabla(\rho^{-1}z)\|^r_{L^2(Q)}\\
				&\hspace{1cm}\leq C\|\rho B\|_{L^2(0,T;H^{-r}(\Omega))} \Big( s^r\|\rho^{-1}z\|_{L^2(Q)}+ \|\rho^{-1}z\|^{1-r}_{L^2(Q)}\|\rho^{-1}\nabla z\|^r_{L^2(Q)}\Big)\\
				&\hspace{1cm}\leq Cs^{r-3/2}\|\rho B\|_{L^2(0,T;H^{-r}(\Omega))} \\
				&\hspace{2cm}\Big( s^{3/2}\|\rho^{-1} z\|_{L^2(Q)}+ \big( s^{3/2} \|\rho^{-1} z\|_{L^2(Q)}\big)^{1-r}\big(s^{1/2}\|\rho^{-1}\nabla z\|_{L^2(Q)}\big)^r\Big)\\
				&\hspace{1cm}\leq  Cs^{r-3/2}\|\rho B\|_{L^2(0,T;H^{-r}(\Omega))} \|z\|_{P_s}
			\end{aligned}$$
			and conclude that the right hand side of \eqref{bilinear equation} is a linear continuous form on $P_s$. 
			The Riesz representation theorem implies the existence of a unique $w\in P_s$ satisfying the formulation \eqref{bilinear equation} and 
					\begin{align}\label{bound-w_s}
			\|w\|_{P_s} \leq C \left(s^{r-3/2} \|\rho B\|_{L^2(0,T;H^{-r}(\Omega))} + s^{-1/2}   \|\rho(0)  u_0\|_{L^2(\Omega)} + s^{-1/2}   \|\rho(0)  u_1\|_{L^2(\Omega)}  \right).
		\end{align}
		Then, set $y= \rho^{-2}Lw$ and $v= s\eta^2\Psi\rho^{-2}\partial_\nu w$. From \eqref{bilinear equation}, the pair $(y, v)$ satisfies 
		\begin{equation}
		\begin{aligned}
		\label{transposition_eq}
			 \intq y L z \dx \dt + \int_{\Sigma}  v \partial_\nu z d\Sigma= & <u_1,\,z(\cdot,0)>_{H^{-1}(\Omega),H_0^1(\Omega)}  - \int_\Omega u_0\, z_t(\cdot,0) \dx \\
			&+ <B, z>_{L^2(0,T; H^{-1}(\Omega)), L^2(0,T;H_0^1(\Omega))}\ \ \forall z \in P_s, 
		\end{aligned}
		\end{equation}
		meaning that $y \in L^2(Q)$ is a solution to \eqref{linear_control_problem} associated with the function $v\in L^2(\Sigma)$ in the sense of transposition. Eventually, using \eqref{bound-w_s}, we get that $\rho y=\rho^{-1} Lw \in L^2(Q)$ and $s^{-1/2}\rho v= s^{1/2} \eta^2\Psi\rho^{-1} \partial_\nu w  \in L^2(\Sigma)$ and deduce \eqref{weighted_estimate_control_traj-frac} for the first and second left hand side terms.
		
		To get \eqref{weighted_estimate_control_traj-frac} for the third and fourth left hand side terms, we remark that  $\rho y \in  L^\infty(0,T;L^2(\Omega))$:  indeed, $y\in\mathcal{C}^0([0,T];L^2(\Omega))\cap \mathcal{C}^1([0,T];H^{-1}(\Omega))$ (see  \cite[Theorem 4.2 p.46]{Lions-1}). 

Moreover, since $\rho^{-1}\in\mathcal{C}^\infty(\overline{Q})$, for any $z\in P_s$, we get 
 \begin{align}\nonumber
	\begin{dcases}
		L(\rho^{-1}z) =\rho^{-1}Lz+\widetilde B_z, \qquad \widetilde B_z:=2\rho^{-1}_tz_t+\rho^{-1}_{tt}z-2\nabla \rho^{-1}\cdot\nabla z-\Delta\rho^{-1}z, \\
		(\rho^{-1}z)_{|\partial\Omega}=0 \\
 	\end{dcases}
\end{align}  
so that \eqref{transposition_eq} rewrites
		\begin{align*}
			&\intq \rho y  L (\rho^{-1} z) \dx \dt + \int_{\Sigma} \rho v  \partial_\nu (\rho^{-1} z) d\Sigma\\
			&=  \langle  \rho(0) u_1, (\rho^{-1} z)(\cdot,0) \rangle_{H^{-1}(\Omega)\times H^1_0(\Omega)} - \int_\Omega \rho(0)  u_0\, (\rho^{-1} z)_t(\cdot,0) \dx  +\int_\Omega \rho(0)  u_0\, \rho^{-1}_ t(0) z (\cdot,0) \dx\\
			&+  \langle  \rho B, \rho^{-1} z \rangle_{L^2(0,T;H^{-1}(\Omega))\times L^2(0,T;H^1_0(\Omega))} +\intq \rho y \widetilde B_z \dx \dt. 
		\end{align*}
Remark that $z\in P_s$ if and only if $\widetilde z:=\rho^{-1} z\in P_s$; therefore, for all $\widetilde  z \in P_s$ and $z=\rho \widetilde z$, 
		\begin{align*}
			&\intq \rho y  L \widetilde  z\dx \dt + \int_{\Sigma} \rho v  \partial_\nu \widetilde  z d\Sigma\\
			&=  \langle  \rho(0) u_1, \widetilde  z(\cdot,0) \rangle_{H^{-1}(\Omega)\times H^1_0(\Omega)} - \int_\Omega \rho(0)  u_0\, \widetilde z_t(\cdot,0) \dx  + \langle  \rho B, \widetilde z\rangle_{L^2(0,T;H^{-1}(\Omega))\times L^2(0,T;H^1_0(\Omega))} \\
			&+\int_\Omega \rho(0)  u_0\, \rho^{-1}_ t(0) z (\cdot,0) \dx+\intq \rho y \widetilde B_z \dx \dt. 
		\end{align*}
Moreover, using that $\vert\nabla \rho^{-1}\vert \leq Cs \rho^{-1}$, $|(\rho^{-1})_t|\leq Cs\rho^{-1}$,  $|(\rho^{-1})_{tt}|\leq Cs^2\rho^{-1}$  and $|\nabla^2 \rho^{-1}|\leq Cs^2\rho^{-1}$, we get the estimates
$$
\left|\int_\Omega \rho(0)  u_0\, \rho^{-1}_ t(0) z (\cdot,0) \dx\right| \le Cs \|\rho(0) u_0\|_{L^2(\Omega)}\| \widetilde  z (\cdot,0)\|_{L^2(\Omega)},
$$
$$
\|\widetilde B_z \|_{L^2(Q)}\le C(s^2\|\widetilde z\|_{L^2(Q)}+s\|\widetilde z_t\|_{L^2(Q)}+s\|\nabla \widetilde z\|_{L^2(Q)})
$$
and thus, since $s\ge 1$ : 
$$\left|\intq \rho y \widetilde B_z \dx \dt\right|\le Cs^2 \|\rho y\|_{L^2(Q)}( \|\widetilde z\|_{L^2(Q)}+\|\widetilde z_t\|_{L^2(Q)}+\|\nabla \widetilde z\|_{L^2(Q)}).
$$
Then, for all $\widetilde  z \in P_s$
\begin{align*}
\left|\intq \rho y  L \widetilde  z\dx \dt\right|
\le &\left| \int_{\Sigma} \rho v  \partial_\nu \widetilde  z d\Sigma\right|+ \left| \langle  \rho(0) u_1, \widetilde  z(\cdot,0) \rangle_{H^{-1}(\Omega)\times H^1_0(\Omega)}\right|+\left| \int_\Omega \rho(0)  u_0\, \widetilde z_t(\cdot,0) \dx \right| \\
			& +\left|\int_\Omega \rho(0)  u_0\, \rho^{-1}_ t(0) z (\cdot,0) \dx\right|+ \left| \langle  \rho B, \widetilde z\rangle_{L^2(0,T;H^{-1}(\Omega))\times L^2(0,T;H^1_0(\Omega))}\right|+\left|\intq \rho y \widetilde B_z\dx \dt\right| \\
\le &\|\rho v\|_{L^2(\Sigma)}\| \partial_\nu \widetilde  z\|_{L^2(\Sigma)}+\|\rho(0) u_1\|_{H^{-1}(\Omega)}\| \widetilde  z(\cdot,0)\|_{H^1_0(\Omega)}\\
&+\|\rho(0) u_0\|_{L^2(\Omega)}\| \widetilde  z_t(\cdot,0)\|_{L^2(\Omega)}+Cs \|\rho(0) u_0\|_{L^2(\Omega)}\| \widetilde  z (\cdot,0)\|_{L^2(\Omega)}\\
&+\|\rho B\|_{L^1(0,T;H^{-1}(\Omega))}\|\widetilde z\|_{L^\infty(0,T;H^1_0(\Omega))}+Cs^2 \|\rho y\|_{L^2(Q)}( \|\widetilde z\|_{L^2(Q)}+\|\widetilde z_t\|_{L^2(Q)}+\|\nabla \widetilde z\|_{L^2(Q)}).
\end{align*}
For any $g\in \mathcal{D}(Q)$, let  $\widetilde  z \in P_s$ be solution of $L\widetilde z =g$ and $(\widetilde z(0),\widetilde z_t(0))=(0,0)$ so that
		$$\| \widetilde z_t \|_{L^\infty(0,T;L^2(\Omega))}+\| \widetilde z \|_{L^\infty(0,T;H^1_0(\Omega))}+\| \partial_\nu \widetilde  z\|_{L^2(\Sigma)}\le C\|g\|_{L^1(0,T;L^2(\Omega))}$$
		and thus, using \eqref{weighted_estimate_control_traj-frac}, for all $0\le r\le 1$, $r\not=1/2$, since $s\ge  1$ : 
		\begin{align*}
\left|\intq \rho y  g\dx \dt\right|
&\le \|\rho v\|_{L^2(\Sigma)}\| \partial_\nu \widetilde  z\|_{L^2(\Sigma)}\\
&+\|\rho B\|_{L^1(0,T;H^{-1}(\Omega))}\|\widetilde z\|_{L^\infty(0,T;H^1_0(\Omega))}+Cs^2 \|\rho y\|_{L^2(Q)}( \|\widetilde z\|_{L^2(Q)}+\|\widetilde z_t\|_{L^2(Q)}+\|\nabla \widetilde z\|_{L^2(Q)})\\
&\le C\left(s^{r+1/2}\|\rho B\|_{L^2(0,T;H^{-r}(\Omega))}+ s^{3/2}  \|\rho(0)  u_0\|_{L^2(\Omega)} + s^{3/2}   \|\rho(0) u_1\|_{H^{-1}(\Omega)}  \right)\|g\|_{L^1(0,T;L^2(\Omega))}.
	\end{align*}
Therefore, $\rho y \in (L^1(0,T;L^2(\Omega))'=L^\infty(0,T;L^2(\Omega))$
and
$$
\|\rho y\|_{L^\infty(0,T;L^2(\Omega))}\le C\left(s^{r+1/2}\|\rho B\|_{L^2(0,T;H^{-r}(\Omega))}+ s^{3/2}  \|\rho(0)  u_0\|_{L^2(\Omega)} + s^{3/2}   \|\rho(0) u_1\|_{H^{-1}(\Omega)}  \right).
$$
Similarly,  for any $g\in \mathcal{D}(Q)$, let $\widetilde  Z \in P_s$ satisfying $L\widetilde Z=g_t$ and $(\widetilde Z(0),\widetilde Z_t(0))=(0,0)$. Then using \cite[(4.19) p.51]{Lions-1}  we obtain :
		\begin{align*}
\left|\intq \rho y  g_t\dx \dt\right|
&\le \|\rho v\|_{L^2(\Sigma)}\| \partial_\nu \widetilde  Z\|_{L^2(\Sigma)}+\|\rho B\|_{L^1(0,T;H^{-1}(\Omega))}\|\widetilde Z\|_{L^\infty(0,T;H^1_0(\Omega))}\\
&+Cs^2 \|\rho y\|_{L^2(Q)}( \|\widetilde Z\|_{L^2(Q)}+\|\widetilde Z_t\|_{L^2(Q)}+\|\nabla \widetilde Z\|_{L^2(Q)})	\\
&\le C\left(s^{1/2+r}\|\rho(s)B\|_{L^2(0,T;H^{-r}(\Omega))}+ s^{3/2}  \|\rho(0)  u_0\|_{L^2(\Omega)} + s^{3/2}   \|\rho(0) u_1\|_{H^{-1}(\Omega)}  \right)\|g\|_{L^1(0,TH^1_0(\Omega))}
	\end{align*}
and thus 
$(\rho y)_t\in L^\infty(0,T;H^{-1}(\Omega))$ and
$$\|(\rho y)_t\|_{L^\infty(0,T;H^{-1}(\Omega))}\le C\left(s^{r+1/2}\|\rho B\|_{L^2(0,T;H^{-r}(\Omega))}+ s^{3/2}  \|\rho(0)  u_0\|_{L^2(\Omega)} + s^{3/2}   \|\rho(0) u_1\|_{H^{-1}(\Omega)}  \right).$$ 
\end{proof}
\begin{remark}\label{remark_J_s}
	The state-control pair $(y,v)$ introduced in \Cref{Lemma_relation_control} is the unique minimizer of the functional 
	\begin{equation}\label{extremal-2}
		J_s(z,u) := s\intq \rho^2 |z|^2 \dx\dt  + \int_{\delta}^{T-\delta}\!\!\!\!\int_{\partial\Omega} \eta^{-2} \Psi^{-1}\rho^{2} |u|^2  \dx\dt
		\end{equation} 
	over 
	$\{ (z,u): z\in L^2(Q), \eta^{-1}\Psi^{-1/2}\rho u\in L^2(\delta,T-\delta; L^2(\partial\Omega)) \text{ solution of }   \eqref{linear_control_problem}   \text{ with } z(\cdot,T)= z_t(\cdot,T)=0 \text{ in }\Omega\}$. We refer to \cite[Section 2]{Munch-Cara-Nicolae}. 
\end{remark}

\begin{remark}\label{coupling_ys_ws}
The controlled state $y= \rho^{-2} L w$ satisfies 
%
$$
\begin{aligned}
		L y = B \text{ in } Q, \qquad y = s \eta^2\Psi \rho^{-2} \partial_\nu w \text{ on } \Sigma, \qquad
	(y(\cdot,0),y_t(\cdot,0)) =(u_0,u_1)\text{ in } \Omega, 
	\end{aligned} 
$$
so that $y\in \C^0([0,T];L^2(\Omega))  \cap \C^1([0,T];H^{-1}(\Omega))$. On the  other hand, the function $w$ satisfies 
$$
\begin{aligned}
	L w = \rho^2 y \text{ in } Q,  \qquad w= 0 \text{ on } \Sigma
\end{aligned}
$$
implying according to \eqref{Carleman-in-0-T} that $(w(\cdot,0),\partial_t w(\cdot,0))\in \boldsymbol{V}$ and 
$w \in   \C^0([0,T]; H^1_0(\Omega))  \cap \C^1([0,T];L^{2}(\Omega))$.
\end{remark}

\subsection{Estimates for the state-control pair with $(u_0,u_1,B)\in \boldsymbol{V}\times L^2(Q)$} \label{Section-regularity-bis}
	
The state-control pair $(y,v)$ given by \Cref{Lemma_relation_control} enjoys additional regularity properties, under additional regularity assumption on the data and introduction of appropriate cut-off function in space and time. Such gain of regularity is known for the wave equation since \cite{DehmanLebeau, ervedoza_zuazua_2010} and more recently \cite{Bhandari_MCSS_2022}. The following first regularity result extends  \cite[Theorem 7]{Bhandari_MCSS_2022} to the multi-dimensional case and is proven in \Cref{sec:appendix}. It will be crucial for the analysis of the semilinear case discussed in \Cref{Controllability_results}.


\begin{prop}\label{Thm-regularity-one}
  For any $x_0\notin\overline{\Omega}$, we assume \eqref{Geometric_Condition} and \eqref{condT}. Let $\eta\in \C_c^1(\mathbb{R})$ and $\Psi\in \mathcal{C}^2(\partial\Omega)$ be cut-off functions satisfying  \eqref{function_Psi} and \eqref{function_eta} respectively. Let  any  $(u_0, u_1)\in \boldsymbol{V}$ and  $B \in  L^2(Q)$  be given. For any $s\geq s_0$, the solution $(y,v)$ of \eqref{linear_control_problem} defined in \Cref{Lemma_relation_control} satisfies 
\begin{equation}
v\in \C^0 ([0,T];H^{1/2}(\partial\Omega)), \quad y\in \C^0([0,T];H^1(\Omega))\cap \C^1([0,T];L^2(\Omega))
\end{equation}
and the following estimate : 
\begin{equation}   
\label{estimate_regular}
\begin{aligned}	
&\|(\rho  y)_t \|_{L^2(Q)} + s^{-1/2} \big\| \frac{1}{ \eta\Psi^{1/2}}(\rho v)_t \big\|_{L^2(\Sigma)}  +s^{-1}\|\nabla (\rho y)\|_{L^2(Q)}+s^{-3/2}\|\rho v\|_{L^\infty(0,T;H^{1/2}(\partial\Omega))}\\
&\hspace{1cm}+s^{-2}(\|\rho  y\|_{\C^0([0,T]; H^1(\Omega))} + \|(\rho y)_t\|_{\C^0([0,T];L^2(\Omega))})\\
&	\hspace{2cm}\leq  C s^{-1/2}
	\left(\|\rho B\|_{L^2(Q)} + \|\rho(0)  u_1\|_{L^2(\Omega)} \right.
	\left.+s   \|\rho(0)  u_0\|_{L^2(\Omega)} +  \|\rho(0) \nabla u_0 \|_{L^2(\Omega)}\right).
	\end{aligned}
\end{equation}
\end{prop}

\subsection{Estimates for the state-control pair with $(u_0,u_1,B)\in H_0^{1-r}(\Omega)\times H^{-r}(\Omega) \times L^2(0,T; H^{-r}(\Omega))$} \label{Section-regularity-Hr}

		\begin{prop}\label{Interpolation}
		  For any $x_0\notin\overline{\Omega}$, we assume \eqref{Geometric_Condition} and \eqref{condT}. Let $\eta\in \C_c^1(\mathbb{R})$ and $\Psi\in \mathcal{C}^2(\partial\Omega)$ be cut-off functions satisfying  \eqref{function_Psi} and \eqref{function_eta} respectively. For any $s\geq s_0$ and $r\in (0,1)$, $r\neq 1/2$, there exists a constant $C>0$ (depending on $s$ and $r$) such that
		\begin{equation}\label{weighted_estimate_control_traj-Hr}
		\begin{aligned}	
			\|\rho\,y\|_{\mathcal{C}^{1}([0,T];H^{-r}(\Omega)) }&+\|\rho\,y\|_{ \C([0,T];H^{1-r}(\Omega)) } +  \left\|\frac{\rho}{\eta\Psi^{1/2}} v\right\|_{H^{1-r}(0,T;L^2(\partial\Omega)) }\\
			&\leq C \left(  \|\rho B\|_{L^2(0, T;H^{-r}(\Omega)) } +  \|\rho(0)  u_0\|_{H^{1-r}(\Omega)} +   \|\rho( 0) u_1\|_{H^{-r}(\Omega)}  \right).
		\end{aligned}
		\end{equation}
	\end{prop}
	
\begin{proof}
Let $\Lambda^0_s: (B, u_0, u_1) \mapsto (y,v)$ be the linear operator with $(y,v)$ the control-state pair given by \Cref{Lemma_relation_control}. Then, from \Cref{Lemma_relation_control} and \Cref{Thm-regularity-one},
$$\Lambda^0_s:\ L^2(0, T;H^{-1}(\Omega)) \times L^2(\Omega)\times H^{ -1}(\Omega)\to (\mathcal{C}^0([0,T];L^2(\Omega))\cap \mathcal{C}^1([0,T];H^{-1}(\Omega)))\times L^2(0,T;L^2(\partial\Omega))  $$ 
and
$$\Lambda^0_s:\ L^2(0, T;L^2(\Omega)) \times H^1_0(\Omega)\times L^2(\Omega)\to ( \C^0 ([0,T];H^1 (\Omega))\cap \mathcal{C}^1([0,T];L^2(\Omega)))\times H^1(0,T;L^2(\partial\Omega))  $$
are  linear continuous. By interpolation, for all $0<\theta<1$, 
	 $\Lambda^0_s$ 			is linear  continuous from $L^2(0, T;(H^{-1}(\Omega),L^2(\Omega))_{\theta}) \times (L^2(\Omega),H^1_0(\Omega))_ \theta \times (H^{ -1}(\Omega),L^2(\Omega))_\theta $ to $( \mathcal{C}^0([0,T];(L^2(\Omega)), H^1(\Omega))_\theta)\cap \C^1([0,T];(H^{-1} (\Omega),L^2(\Omega))_\theta)) \times (L^2(0,T;L^2(\partial\Omega)) ,H^1(0,T;L^2(\partial\Omega)))_\theta $.
	 
	 Thus, for $\theta=1-r$, $0<r<1$, $r\not=1/2$,  $\Lambda^0_s$ 	 is linear, continuous from $L^2(0, T;H^{-r}(\Omega)) \times H^{1-r}_0(\Omega)\times H^{ -r}(\Omega)$ to $(\C^0([0,T];H^{1-r}(\Omega))\cap \mathcal{C}^{1}([0,T];H^{-r}(\Omega)))\times H^{1-r}(0,T;L^2(\partial\Omega))  $. 
\end{proof}

\section{Proof of Theorem \ref{main_th} }\label{Controllability_results}

In this section, we prove the controllability of the semilinear equation $\eqref{main_control_problem}$. In this respect,  for all $s\geq s_0$ and for all $\hat{y}$ in an appropriate subset $\C(s)$ of $L^\infty(0,T;L^2(\Omega))$, we consider the linearized boundary control problem
\begin{equation}\label{linearized_model}
Ly  = - f(\yy) \text{ in } Q, \quad y =  v\,1_{\Gamma_0} \text{ on } \Sigma, \quad (y(\cdot,0),y_t(\cdot,0))=(u_0,u_1) \text{ in } \Omega
\end{equation}
such that $(y(\cdot, T), y_t(\cdot, T))=(0,0)$ in $\Omega$. The existence of a controlled trajectory $y\in L^2(Q)$ is guaranteed by  \Cref{Lemma_relation_control} with a source term $B=-f(\widehat{y}) \in L^2(0,T;H^{-1}(\Omega))$.

\subsection{First part of Theorem \ref{main_th}} \label{sec:FixedPointPart1}

 Without restriction, we assume that $f\in \C^0(\mathbb{R})$ satisfies \ref{H1} for some $1< p<3/2$. For any $s\geq s_0\geq 1$, we introduce the class $\C(s)$ defined as the closed convex subset of $L^\infty(0,T; L^2(\Omega)) $  
\begin{equation}\label{closed_convex_set_Schauder}
\C(s): = \Big\{  y\in L^\infty(0,T; L^2(\Omega)) :\   \|\rho y \|_{L^2(Q)} \leq s, \ \|\rho  y \|_{L^\infty(0,T; L^2(\Omega))} \leq s^3\Big\}.
\end{equation}
We prove the existence of $s$ and of a fixed point of the operator
\begin{equation}\label{operator_fixed_point}
\Lambda_s : { \C(s)} \rightarrow { \C(s)} \qquad \widehat{y} \mapsto y
\end{equation}
where $y$ is a solution of the null controllability problem \eqref{linearized_model} associated with the control $v$ given by \Cref{Lemma_relation_control} for $B=-f(\widehat{y})$. 
We employ the Schauder theorem and we prove that : $\boldsymbol{i)}$ for $\beta^\star>0$ small enough, there exists $s\geq s_0$ large enough such that $\C(s)$ is stable under the map $\Lambda_s$ (\Cref{sec-stable-Schauder}); $\boldsymbol{ii)}$ $\Lambda_s(\C(s))$ is a relatively compact subset of $\C(s)$ for the norm $\|.\|_{L^\infty(0,T; L^2(\Omega))}$ (\Cref{Relative_Compactness_Schauder}); $\boldsymbol{iii)}$ $\Lambda_s$ is a continuous map in $\C(s)$ with respect to the $L^\infty(0,T; L^2(\Omega))$ norm (\Cref{Continuity-Schauder}).
This ensures the existence of a fixed point for $\Lambda_s$, which is a controlled trajectory for \eqref{main_control_problem}.

\subsubsection{Estimates  of $ \Lambda_s(\yy)$}

\begin{lemma}\label{estimation_lemma1-H}
Assume that $f\in \C^0(\R)$ satisfies \ref{H1}. For any $s\geq s_0$ and $\yy\in \C(s)$, there exists $C>0$ such that
$$
\Vert \rho f(\widehat y)\Vert_{L^2(0,T; H^{p-3/2}(\Omega))}\leq C \big(\alpha_1 e^{-s}T^{1/2}\vert \Omega\vert^{1/2}+\alpha_2 s+\beta^\star c^p s^{1+p}\big)
$$
with $c=\Vert \phi\Vert_{L^\infty(Q)}$.
\end{lemma}
\begin{proof} We infer, for $1/2>r= 3/2-p> 0$ and  $1\le p^\star=\frac{2d}{d+3-2p}<2$ that (using that $\rho\leq e^{-s}$)
$$
\begin{aligned}
\Vert \rho f(\widehat y)\Vert_{L^2(0,T; H^{-r}(\Omega))}&\leq C  \Vert \rho \alpha_1 + |\rho \widehat y|\left( \alpha_2 + \beta^\star\ln_+^{p}(\widehat y) \right)\Vert_{L^2(0,T; L^{p^\star}(\Omega))}\\
& \leq C (\alpha_1 e^{-s}\sqrt{T}\vert \Omega\vert^{1/2}+\alpha_2\Vert \rho \widehat y\Vert_{L^2(Q)}+\beta^\star  \Vert \rho \widehat y \ln_+^p(\widehat y)\Vert_{L^2(0,T; L^{p^\star}(\Omega))}.
\end{aligned}
$$
Now, using that $\rho^{-1}\leq e^{cs}$ and $0\le \ln_+^p(\widehat y)\le C(\ln_+^p \rho^{-1}+ \ln_+^{p}(\rho \widehat y))\le  C((cs)^p+ \ln_+^p(\rho \widehat y))$, we get that
$$
\Vert \rho \widehat y \ln_+^{p}(\widehat y)\Vert_{L^2(0,T;L^{p^\star}(\Omega))}\leq C\big((cs)^p\Vert \rho \widehat y\Vert_{L^2(Q)}+ \Vert\rho \widehat y\ln_+^{p}(\rho \widehat y)\Vert_{L^2(0,T;L^{p^\star}(\Omega))}\big).
$$
But, for all $0<\varepsilon<\frac{2-p^\star}{p^\star}=\frac{2r}{d}=\frac{3-2p}{d}$, there exists $C_\varepsilon>0$ such that
\begin{equation}
\begin{aligned}\label{estim-L2H1}
\|\rho \widehat y \ln^p_+ (\rho \widehat y)\|_{L^2(0,T;L^{p^\star}(\Omega))}
&\le C_\varepsilon \| |\rho \widehat y|^{1+\varepsilon}\|_{L^2(0,T;L^{p^\star}(\Omega))}
\le C_\varepsilon   \| \rho \widehat y \|_{L^2(Q)}\| \rho \widehat y \|_{L^\infty(0,T;L^2(\Omega))}^{\varepsilon}. 
\end{aligned}
\end{equation}
Combining the above inequalities with $\varepsilon<\min(\frac p3,\frac{3-2p}{d})$ and using that  $\yy\in \C(s)$ we get the result. 
\end{proof}

\begin{prop}\label{Proposition-linear_to_nonlinear_Schauder}
For any $x_0\notin\overline{\Omega}$, we assume \eqref{Geometric_Condition} and \eqref{condT}.
	Assume $(u_0,u_1)\in \boldsymbol{H}$ and that $f\in \CC^0(\mathbb{R})$ satisfies \ref{H1}.  For  $s\geq s_0$ and any $\yy\in \C(s)$, the solution $y=\Lambda_s(\yy)$ to the linearized controlled system \eqref{linearized_model} with control $v$ satisfies the following estimates:
\begin{multline}\label{linearized-estimate-Schauder-1}
			\|\rho y \|_{L^2(Q)} + s^{-1/2} \biggl\| \frac{\rho}{ \Psi^{1/2}\eta}v \biggr\|_{L^2(\Sigma)} +s^{-2}\|\rho y\|_{L^\infty(0,T;L^2(\Omega))}+s^{-2}\|(\rho y)_t\|_{L^\infty(0,T;H^{-1}(\Omega))}\\
		\le C s\Big(  s^{-p}\alpha_2  + \beta^\star c^{p}+ e^{-s}\big(s^{-p-1}   \alpha_1 T^{1/2}|\Omega|^{1/2}  + s^{-3/2}  (\|u_0\|_{L^2(\Omega)} + s\|u_1\|_{H^{-1}(\Omega)})  \big)\Big). 	
\end{multline}
\end{prop}
	
\begin{proof}
The map $L: \C^1(\overline{\Omega})\times  H^{-1}(\Omega)\to  H^{-1}(\Omega)$,  $(\rho(0) ,u_1)\to \rho(0) u_1$ is bilinear continuous (see \cite[Lemma 3, p. 1097]{J.Simon-NonHomo-NS}) so that $\|\rho(0)u_1\|_{H^{-1}(\Omega)}\le C\|\rho(0)\|_{ \C^1(\overline{\Omega})}\|u_1\|_{H^{-1}(\Omega)}\le C se^{-s}\|u_1\|_{H^{-1}(\Omega)}$. 
The estimates follow from \Cref{Lemma_relation_control} with $r=3/2-p$ and  \Cref{Thm-regularity-one} with $B=-f(\widehat{y})$ and \Cref{estimation_lemma1-H}.
\end{proof}


\subsubsection{Stability of the class $\C(s)$}\label{sec-stable-Schauder}

\begin{lemma}\label{Stability_class_Schauder}
	Assume that $f\in \C^0(\mathbb{R})$ satisfies \ref{H1} 
	with $\beta^\star$ small enough. Then, there exists $s\geq s_0$ large enough such that
	$\Lambda_s\big(\C(s) \big) \subset \C(s)$.
	\end{lemma} 

\begin{proof}
For  any $s\geq s_0$ and $\yy\in \mathcal{C}(s)$, let $y:=\Lambda_s(\yy)$. From \eqref{linearized-estimate-Schauder-1}, we obtain that $\limsup_{s\to+\infty} s^{-1}\|\rho y \|_{L^2(Q)}\le \beta^\star C   c^p$ 
and 
$\limsup_{s\to+\infty} s^{-3}\|\rho y \|_{L^\infty(0,T;L^2(\Omega))}\le \beta^\star C   c^p$. Therefore, if $\beta^\star>0$ is small enough so that $\beta^\star C   c^p <1$, then for any $s\geq s_0$ large enough, $y\in \mathcal{C}(s)$.
\end{proof}

\begin{remark}
The lower bound for $s\geq s_0$ is related to the norm of the initial data $(u_0,u_1)\in \boldsymbol{H}$: in view of \Cref{Proposition-linear_to_nonlinear_Schauder}, the stability of $\Lambda_s$
requires that $(C   \| u_0\|_{L^2(\Omega)} +Cs^{-1}( \alpha_1 T^{1/2}|\Omega|^{1/2} + \| u_1 \|_{H^{-1}(\Omega)})) 
e^{-s}<s^{3/2}$. Therefore, the lower bound for $s$ much be chosen as depending logarithmically on $\Vert u_0,u_1\Vert_{\boldsymbol{H}}$.
\end{remark}

\subsubsection{Relative compactness of the set $\Lambda_s(\C(s))$}\label{Relative_Compactness_Schauder}

\begin{lemma}\label{RelativeCompactness}
Under the assumptions of Lemma \ref{Stability_class_Schauder}, $\Lambda_s(\C(s))$ is a relatively compact subset of $\C(s)$ for the $L^\infty(0,T;L^2(\Omega))$ norm.
\end{lemma}
\begin{proof}
Let $(y_n)_{\in\N}$ be a sequence of $\Lambda_s(\C(s))$.  Then, there exists  $(\widehat y_n)_{\in\N}$, sequence of $ \C(s)$ such that, for all $n\in\N$, $y_n=\Lambda_s(\widehat y_n)$.
Remark that the sequence $(y_n)_{\in\N}$ does not enjoy enough regularity to use classical compactness results. However, for all $n\in\N$, $y_n-y_0$ is solution of \eqref{bilinear equation} for $B=f(\yy_n)-f(\yy_0 )$, $u_0=u_1=0$ and thus, we deduce from Proposition \ref{Interpolation}, for $r=3/2-p$ that $\rho(y_n-y_0)\in  \mathcal{C}^{1}(0,T;H^{-r}(\Omega))\cap \C^0([0,T];H^{1-r}(\Omega)))$, $\rho(v_n-v_0)\in  H^{1-r}(0,T;L^2(\partial\Omega))  $
and 
$$		\begin{aligned}	
			\|\rho (y_n-y_0)\|_{\mathcal{C}^{1}(0,T;H^{-r}(\Omega)) }&+\|\rho (y_n-y_0)\|_{ \C^0([0,T];H^{1-r}(\Omega)) } +  \left\|\frac{\rho}{\eta\Psi^{1/2}} (v_n-v_0)\right\|_{H^{1-r}(0,T;L^2(\partial\Omega)) }\\
			&\leq C(s)   \|\rho(f(\yy_n)-f(\yy_0))\|_{L^2(0, T;H^{-r}(\Omega)) } .
		\end{aligned}
$$
This  gives with \Cref{estimation_lemma1-H} that $(\rho(y_n-y_0))_{n\in\N}$ is a bounded sequence of $\mathcal{C}^{1}(0,T;H^{-r}(\Omega)) \cap\C^0([0,T];H^{1-r}(\Omega))$. Thus, since 
 $\mathcal{C}^{1}(0,T;H^{-r}(\Omega)) \cap\C^0([0,T];H^{1-r}(\Omega))\hookrightarrow L^\infty(0,T;L^2(\Omega))$ is compact  (see \cite{Simon1}, Corollary 5 p.86), there exist a subsequence $(\rho (y_{n_k}-y_0))_{n_k\in\N}$ and $z\in L^\infty(0,T;L^2(\Omega))$ such that $\rho(y_{n_k}-y_0)\to z$ in $L^\infty(0,T;L^2(\Omega))$. Therefore, $\rho y_{n_k}\to y=z+\rho y_0$ in $L^\infty(0,T;L^2(\Omega))$. 
 Since $\rho y_{n_k}\in \mathcal{C}(s)$ for all $n_k$, $ y=z+\rho y_0\in  \mathcal{C}(s)$. Thus $\Lambda_s(\C(s))$ is a relatively compact subset of $\C(s)$ for the $L^\infty(0,T;L^2(\Omega))$ norm.
 \end{proof}

\subsubsection{Continuity of the map $\Lambda_s$ in $\C(s)$}\label{Continuity-Schauder}

\begin{lemma}\label{ContinuityMap}
Assume that $f\in \mathcal{C}^0(\mathbb{R})$ satisfies \ref{H1}. Then, the map $\Lambda_s : \C(s)\rightarrow \C(s)$ is continuous for the $L^\infty(0,T;L^2(\Omega))$ norm.
\end{lemma}

\begin{proof}
Let $(\yy_n)_{n\in\mathbb{N}}$ be a sequence in $\C(s)$ such that $(\yy_n)_{n\in\mathbb{N}}$ converges to $\yy\in \C(s)$ with respect to the $L^\infty(0,T;L^2(\Omega))$ norm. Let $y_n := \Lambda_s(\yy_n)$ for all $n\in \mathbb{N}$. Since $\yy_n\to \yy$ in $L^\infty(0,T;L^2(\Omega))$, $\yy_n\to \yy$ in $L^2(Q)$ and 
there exist a subsequence $(\yy_{n_k})_{n_k\in\mathbb{N}}$  and $z\in L^2(Q)$ such that $\yy_{n_k}\to \yy$ a.e and $|\yy_{n_k}|\le z$ a.e for all $n_k$.

Thus, since $f\in\mathcal{C}^0(\R)$,   $f(\yy_{n_k})\to f(\yy)$ a.e. 
Let us  choose $\varepsilon<1$ small enough so $L^{\frac{2}{1+\varepsilon}}(\Omega) \hookrightarrow H^{-1}(\Omega)$. For all $n_k\in\N$, a.e: 
$$ \begin{aligned}
|f(\yy_{n_k})|
&\le  \alpha_1 + |\yy_{n_k}|\left( \alpha_2 + \beta^\star\ln_+^{p}|\yy_{n_k}| \right)\le  \alpha_1 + \alpha_2  |\yy_{n_k} |+
  \beta^\star\ln_+^{p}|\yy_{n_k}|  |\yy_{n_k}|\\
  &\le  \alpha_1 + \alpha_2  |z|+
  \beta^\star\ln_+^{3/2}|z|  |z|\le  \alpha_1 + \alpha_2  |z|+
  C_\varepsilon \beta^\star |z|^{1+\varepsilon} .
  \end{aligned}$$
and $\alpha_1 + \alpha_2  |z|+
  C\beta^* |z|^{1+ \varepsilon}\in L^{\frac2{1+ \varepsilon}}(Q) $, since $z\in L^2(Q)$ and thus, we deduce from the dominated convergence theorem that $f(\yy_{n_k})\to f(\yy)$ in $L^{\frac2{1+ \varepsilon}}(Q) $. 
  
  In fact,  $f(\yy_{n })\to f(\yy)$ in $ L^{\frac2{1+ \varepsilon}}(Q)$  (and thus $f(\yy_{n })\to f(\yy)$ in $L^q(0,T;L^{\frac2{1+ \varepsilon}}(\Omega))$ for all $1\le q<+\infty$ since $(f(\yy_n))_{n\in\N}$ is bounded in $L^\infty(0,T;L^{\frac2{1+ \varepsilon}}(\Omega))$). If not, there exists  $\varepsilon_1>0$ and a subsequence $(\yy_{n_k})_{n_k\in\mathbb{N}}$ of $(\yy_n)_{n\in\mathbb{N}}$ such that $\|f(\yy_{n_k})-f(\yy)\|_{ L^{\frac2{1+ \varepsilon}}(Q)}\ge \varepsilon_1$.
  But, arguing as previouly, there exist a subsequence $(\yy_{n_{k'}})_{n_{k'}\in\mathbb{N}}$  of  $(\yy_{n_k})_{n_k\in\mathbb{N}}$ such that $f(\yy_{n_{k'}})\to f(\yy)$ in $ L^{\frac2{1+ \varepsilon}}(Q)$ which leads to a contradiction.
  
We denote by $v_n$ and $v$ the associated control to $y_n$ and $y$ respectively and we have, by definition of the operator $\Lambda_s$, that $(y_n, v_n) = \rho^{-2}(Lw_n, s\eta^2\Psi\partial_\nu w_n)$ and  $(y, v) = \rho^{-2}\left(Lw, s\eta^2\Psi\partial_\nu w\right)$ with $w_n$ and $w$ solution of \eqref{bilinear equation} associated with $(u_0, u_1, -f(\widehat{y}_n))$ and $(u_0, u_1, -f(\widehat{y}))$ respectively. In particular, $z_n := y_n - y = \Lambda_s(\yy_n) - \Lambda_s(\yy)$ satisfies $z_n = \rho^{-2} L(w_n - w)$ with $w_n - w$ solution of \eqref{bilinear equation} associated with data $\left(0,0, f(\widehat{y})-f(\widehat{y}_n)\right)$. Thus, using that $L^2(0,T;L^{\frac2{1+ \varepsilon}}(\Omega)) \hookrightarrow L^2(0,T;H^{-1}(\Omega))$, estimate \eqref{weighted_estimate_control_traj-frac} with $B=f(\yy_n)-f(\yy )$, $u_0=u_1=0$ implies
\begin{equation}\label{A}
\begin{aligned}
\|\rho(y_n-y)\|_{L^\infty(0,T;L^2(\Omega))} 
&\leq  C s^{3/2}\| \rho(f(\yy_n)-f(\yy ))\|_{L^2(0,T;H^{-1}(\Omega))}  \\
&\leq  C s^{3/2}\| \rho(f(\yy_n)-f(\yy ))\|_{L^2(0,T;L^{\frac2{1+ \varepsilon}}(\Omega))} 
\end{aligned}
\end{equation}
and thus  $y_n \rightarrow y$ as $n\rightarrow +\infty$ in $L^\infty(0,T;L^2(\Omega))$.
\end{proof}

\subsection{Second part of Theorem \ref{main_th}}\label{sec:FixedPointPart2}

We now assume that $f\in\mathcal{C}^1(\mathbb{R})$ satisfies \ref{asymptotic_behavior_prime_p} with $0\le p<3/2$;  $f$ then satisfies \ref{H1} for $\alpha_1=f(0) $ and $\alpha_2=\alpha$ so that results of \Cref{sec:FixedPointPart1} remain true. We are going to show that $\Lambda_s:\C(s)\to\C(s)$ is a contracting mapping of the complete space $(\C(s),d)$ with $d:\C(s)\times \C(s)\to \R$  defined by $d(y,z):=\Vert \rho(y-z)\Vert_{L^2(Q)}$. The Banach fixed point theorem will ensure the existence of a unique fixed point of $\Lambda_s$ which is a controlled trajectory for the semilinear problem \eqref{main_control_problem}. 
\begin{prop}\label{Prop-Banach_contraction} 
	Assume that  there exists $0\le p<3/2$ such that  $f\in \C^1(\R)$ satisfies \ref{asymptotic_behavior_prime_p} with $\beta^\star$ and $s$ as in \Cref{Stability_class_Schauder}. Then, $\Lambda_s$ is a contraction mapping from $(\C(s),d)$ into itself.
 \end{prop}

\begin{proof} Without restriction, we assume that $p> 1$. Let $\yy_1, \yy_2 \in \C(s)$. From \eqref{weighted_estimate_control_traj-frac}, we get that, for all $0\le r< 1/2$ , $d(\Lambda_s (\yy_2),\Lambda_s(\yy_1))\leq C_r s^{r-3/2} \| \rho(f(\yy_2)  - f(\yy_1))\|_{L^2(0,T;H^{-r}(\Omega))}$. Let  $r=3/2-p>0$. There exists $1\le q<2$ such that $L^q(\Omega) \hookrightarrow  H^{-r}(\Omega)$. We then have
\begin{align*}
d(\Lambda_s (\yy_2),\Lambda_s(\yy_1))
\leq C_r s^{r-3/2} \| \rho(f(\yy_2)  - f(\yy_1))\|_{L^2(0,T;L^q(\Omega))}.
	\end{align*}
But, for all $(m_1,m_2)\in\R^2$ there exists $\bar c$ such that  
$$|f(m_1)-f(m_2)|\le |m_1-m_2||f'(\bar c)|	\le  |m_1-m_2| (\alpha+\beta^\star \ln_+^p|\bar c| ) \le  |m_1-m_2| (\alpha+\beta^\star\ln_+^p( |m_1|+|m_2|) )$$
and therefore, using that $0\le \ln^p_+ \rho \le c^ps^p$ and that $p=3/2-r$, we get
\begin{equation}
\label{eqcondS}
\begin{aligned}
d(\Lambda_s (\yy_2),\Lambda_s(\yy_1)) 
\leq& Cs^{-p} \| (\alpha+\beta^\star\ln_+^p(|\yy_1|+|\yy_2|))\rho(\yy_2  -  \yy_1)\|_{L^2(0,T;L^q(\Omega))}\\
\leq& C s^{-p} \| (\alpha+\beta^\star\ln_+^p(|\yy_1|+|\yy_2|))\|_{L^ \infty(0,T;L^a(\Omega))}\, d(\yy_2,\yy_1)\\
\leq& C s^{-p} \big(\alpha+\beta^\star c^ps^p+ \beta^\star\Vert\ln_+^p(\rho(|\yy_1|+|\yy_2|)))\|_{L^ \infty(0,T;L^a(\Omega))}\big)d(\yy_2,\yy_1)
\end{aligned}
\end{equation}
with $a$ such that $1/q=1/2+1/a$. Now, using that, for $\varepsilon=\inf\{\frac2a, \frac p3\}$
$$
\begin{aligned}
\Vert\ln_+^p(\rho(|\yy_1|+|\yy_2|))\|_{L^ \infty(0,T;L^a(\Omega))}&\leq C\big( \|   (\rho \yy_1)^{ \varepsilon  } \|_{L^ \infty(0,T;L^a(\Omega))}+  \|   (\rho \yy_2)^{\varepsilon  } \|_{L^ \infty(0,T;L^a(\Omega))}\big)\\
&\leq C\big( \|    \rho \yy_1  \|_{L^ \infty(0,T;L^2(\Omega))}^ \varepsilon +  \|   \rho \yy_2  \|_{L^ \infty(0,T;L^2(\Omega))}^ \varepsilon\big)\leq C s^p
\end{aligned}
$$
we infer that 
\begin{equation}\label{Banach_contraction} 
d(\Lambda_s (\yy_2),\Lambda_s(\yy_1)) \leq C(s^{-p} \alpha+\beta^\star c^p)d(\yy_2,\yy_1).	
\end{equation}
The contraction property follows for $s\geq s_0$ large enough and $\beta^\star>0$ small enough.
\end{proof}

\begin{remark}\label{remark-v} This also proves that 
$s^{-1/2}\left\|\frac{\rho}{\eta\Psi^{1/2}} (v_2-v_1)\right\|_{L^2(\Sigma)}
	\leq 
 C(s^{-p} \alpha+\beta^\star c^p)  d(\yy_2,\yy_1)$, for all $\yy_1, \yy_2 \in \C(s)$
where $v_1, v_2$ are the associated controls.
\end{remark}

As a corollary of the Banach fixed point theorem, the contraction property of the operator $\Lambda_s$ for $\beta^\star$ small enough and $s$ large enough allows to define a convergent sequence $(y_k,v_k)_{k\in \N}$ to a controlled pair for \eqref{main_control_problem} and prove the following more precise version of the second item of Theorem \ref{main_th}.

\begin{theorem}\label{banach-Convergence} 
	Let $(u_0, u_1) \in \boldsymbol{H}$.  Assume that $f$ is $\C^1(\R)$ and  satisfies \ref{asymptotic_behavior_prime_p} for some $0\le p<3/2$ with $\beta^\star$ small enough and $s$ as in \Cref{Stability_class_Schauder}. Then, for any $y_0 \in \C(s)$, the sequence $(y_k)_{k\in\mathbb{N}^\star} \subset \C(s)$ given by $y_{k+1}=\Lambda_s(y_k)$, $k\geq 0$,  
(where $\Lambda_s$ is defined by \eqref{operator_fixed_point}) together with the associated sequence of controls $(v_k)_{k\in\mathbb{N}^\star}\subset L^2(\Sigma))$ strongly converges in $\C^0([0,T]; L^2(\Omega))\times L^2(\Sigma)$ to a controlled solution for \eqref{main_control_problem}. The convergence is at least linear with respect to the distance $d$.
\end{theorem}
\begin{proof}
	The convergence of the sequence $(y_k)_{k\in\mathbb{N}}$ toward $y=\Lambda_s(y)\in \mathcal{C}(s)$ with linear rate follows from the contraction property of $\Lambda_s$: 
	$d(y,y_k) = d(\Lambda_s(y),\Lambda_{s}(y_{k-1})) \leq (C(s^{-p} \alpha+\beta^\star c^p))^k d(y,y_0)$ for all $k\geq 0$ deduced from \eqref{Banach_contraction}.
	
Let now $v\in H^1(0,T;L^2(\partial\Omega ))\cap \  \C^0 ([0,T];H^{1/2}(\partial\Omega ))$ be associated with $y$ so that $y-y_k$  satisfies, for every $k\in \mathbb{N}^\star$ 
$$ 
\left\{
	\begin{aligned}
	& L(y-y_k)   = -\big( f(y)-  f(y_{k-1})\big)  & \text{in } Q , \\
	& y-y_k = v-v_k, & \text{on } \Sigma,\\
	&((y-y_k)(\cdot,0),(y-y_k)_t(\cdot,0)) =(0,0) & \text{in } \Omega, \\
	 & 	((y-y_k)(\cdot,T),(y-y_k)_t(\cdot,T)) =(0,0) & \text{in } \Omega.
	\end{aligned}
	\right.
$$ 
From Remark \ref{remark-v}, we deduce that $\big\|\frac{\rho}{\eta\Psi^{1/2}} (v-v_k ) \big\|_{L^2(\Sigma)}  \leq s^{1/2} C(s^{-p}\alpha + \beta^\star c^p)d(y,y_{k-1})$	
and the convergence at a linear rate of the sequence $(v_k)_{k\in\mathbb{N}^\star}$ toward a control $v$ for \eqref{main_control_problem}.

Now, since  $f\in \mathcal{C}^1(\mathbb{R})$ satisfies \ref{H1}, using \eqref{weighted_estimate_control_traj-frac} for $r=3/2-p>0$ and \eqref{Banach_contraction}, we get	
$$
\begin{aligned}
 \|(\rho (y-y_k))_t\|_{\C^0([0,T];H^{-1}(\Omega))}+ \|\rho  (y-y_k)\|_{\C^0([0,T]; L^2(\Omega))}   
&\leq  
C s^{1/2+r}\| \rho(f(y)-f(y_{k-1}))\|_{L^2(0,T;H^{-r}(\Omega))} \\
&\le C  s^{2 } C(s^{-p} \alpha+\beta^\star c^p)\|\rho(y  -  y_{k-1})\|_{L^2(Q)}\\
&\leq Cs^2(C(s^{-p} \alpha+\beta^\star c^p))^k d(y,y_0).
\end{aligned}
$$
It follows that $y_k\to y$ in $\C^0([0,T]; L^2(\Omega))\cap \C^1([0,T];  H^{-1}(\Omega))$.
\end{proof}
\begin{remark}\label{speed_s}
Assume that exists $0\leq p<3/2$ such that $\lim\limits_{|r|\to +\infty} \frac{|f^{\prime}(r)|}{\ln^{p}_+|r|}  =0$, i.e. that $\beta^\star$ is arbitrarily small in \ref{asymptotic_behavior_prime_p}. Then, \eqref{Banach_contraction} shows that the constant of contraction of $\Lambda_s$ behaves like $s^{-p}$.
\end{remark}
\begin{remark}\label{p=troisdemi}
Assume $d\leq 3$ and $(u_0,u_1)\in H^2(\Omega)\cap H_0^1(\Omega) \times H_0^1(\Omega)$. Then, one may prove (repeating one step more the arguments of the Appendix) that the optimal state $y$ given in \Cref{Lemma_relation_control} belongs to $L^\infty(Q)$. This allows (following \cite{Bhandari_MCSS_2022}) to reach the value $p=3/2$ in \ref{asymptotic_behavior_prime_p}. We refer to \cite{Claret_thesis} together with numerical illustrations. 
\end{remark}

\subsection{Third part of Theorem \ref{main_th}} \label{sec:FixedPointPart3}

Assuming the initial data in $\boldsymbol{V}$, we follow the arguments of \Cref{sec:FixedPointPart1} and get that the uniform controllability holds true under the condition \ref{H1} with $p=3/2$.
For any $s\geq s_0$, we introduce the class $\widetilde\C(s)$ defined as the closed convex subset of $L^{2}(Q) $  
\begin{equation}\label{closed_convex_set_Schauder_tilde}
\widetilde\C(s): = \Big\{  y\in H^1(Q) :\   \|\rho y \|_{L^2(Q)} \leq s, \ \|(\rho  y)_t \|_{L^2(Q)} \leq s^2,\ \|\nabla(\rho y) \|_{L^2(Q)} \leq s^3\Big\}.
\end{equation}
\subsubsection{Estimates  of $ \Lambda_s(\yy)$}

\begin{lemma}\label{estimation_lemma1-bis}
Assume that $f\in \C^0(\R)$ satisfies \ref{H1} with $p=3/2$. For any $s\geq s_0$ and $\yy\in \widetilde\C(s)$, there exists $C>0$ such that 
$$
\Vert \rho f(\widehat y)\Vert_{L^2(Q)}\leq C \biggl(\alpha_1 e^{-s}T^{1/2}\vert \Omega\vert^{1/2}+\alpha_2 s+\beta^\star (cs)^{3/2}s\biggr)
$$
with $c=\Vert \phi\Vert_{L^\infty(Q)}$.
\end{lemma}
\begin{proof} Using that $\rho\leq e^{-s}$, we infer that 

$$
\begin{aligned}
\Vert \rho f(\widehat y)\Vert_{L^2(Q)}&\leq C  \Vert \rho \alpha_1 + |\rho \widehat y|\left( \alpha_2 + \beta^\star\ln_+^{3/2}(\widehat y) \right)\Vert_{L^2(Q)}\\
& \leq C (\alpha_1 e^{-s}T^{1/2}\vert \Omega\vert^{1/2}+\alpha_2\Vert \rho \widehat y\Vert_{L^2(Q)}+\beta^\star  \Vert \rho \widehat y \ln_+^{3/2}(\widehat y)\Vert_{L^2(Q)}.
\end{aligned}
$$
Now, since, $0\le \ln_+^{3/2}(\widehat y)\le C(\ln_+^{3/2} \rho^{-1}+ \ln_+^{3/2}(\rho \widehat y))\le  C((cs)^{3/2}+ \ln_+^{3/2}(\rho \widehat y))$, we get that
$$
\Vert \rho \widehat y \ln_+^{3/2}(\widehat y)\Vert_{L^2(Q)}\leq C\big((cs)^{3/2}\Vert \rho \widehat y\Vert_{L^2(Q)}+ \Vert\rho \widehat y\ln_+^{3/2}(\rho \widehat y)\Vert_{L^2(Q)}\big).
$$
But, for all $0<\varepsilon\le 4/d$, there exists $C_\varepsilon>0$ such that
\begin{equation}
\begin{aligned}\label{estim-L2H1-bis}
\|\rho \widehat y \ln^{3/2}_+ (\rho \widehat y)\|_{L^2(Q)}
&\le C_\varepsilon \| |\rho \widehat y|^{1+\frac\varepsilon2}\|_{L^2(Q)}=C_\varepsilon \| \rho \widehat y\|_{L^{2+\varepsilon}(Q)}^{1+\frac\varepsilon2}
\le C_\varepsilon   \| \rho \widehat y \|^{1-\frac{\varepsilon(d-2)}4}_{L^\infty(0,T;L^2(\Omega))}\| \nabla(\rho \widehat y) \|_{L^2(Q)}^{\frac{\varepsilon d}4} \end{aligned}
\end{equation}
and since $\yy \in \widetilde\C(s)$, we have $\|\rho \yy\|_{L^\infty(0,T;L^2(\Omega))} \le  \sqrt{2}\|\rho \yy\|^{1/2}_{L^2(Q)}\|(\rho  \yy)_t\|^{1/2}_{L^2(Q)}\le \sqrt{2}s^{3/2}$.
Thus
$$\begin{aligned}
\|\rho \widehat y \ln^{3/2}_+ (\rho \widehat y)\|_{L^2(Q)}
&\le C_\varepsilon  s^{\frac32-\frac{3\varepsilon(d-2)}8}s^{\frac{3\varepsilon d}4}=C_\varepsilon  s^{\frac32+\frac{3\varepsilon(d+2)}8}.
\end{aligned}
$$
In particular for $\varepsilon =\frac4{3(d+2)}<\frac4d$ this gives
$\|\rho \widehat y\ln^{3/2}_+ (\rho \widehat y)\|_{L^2(Q)}\le C  s^2$. Combining the above inequalities, we get the result. 
\end{proof}

\begin{prop}\label{Proposition-linear_to_nonlinear_Schauder-bis}
For any $x_0\notin\overline{\Omega}$, we assume \eqref{Geometric_Condition} and \eqref{condT}.
	Assume $(u_0,u_1)\in \boldsymbol{V}$. Assume that $f\in \CC(\mathbb{R})$ satisfies \ref{H1} with $p=3/2$. For  $s\geq s_0$ and for all  $\yy\in \widetilde\C(s)$, the solution $y=\Lambda_s(\yy)$ to the linearized controlled system \eqref{linearized_model} with control $v$ satisfies $y\in\C^0([0,T];H^1(\Omega))\cap \C^1([0,T];L^2(\Omega))$ and the following estimates:
\begin{multline}\label{linearized-estimate-Schauder-1-BIS}
			\|\rho y \|_{L^2(Q)} + s^{-1/2} \biggl\| \frac{\rho}{ \Psi^{1/2}\eta}v \biggr\|_{L^2(\Sigma)} +s^{-1}\|(\rho y)_t\|_{L^2(Q)} + s^{-3/2} \biggl\| \frac{1}{ \Psi^{1/2}\eta}(\rho v)_t \biggr\|_{L^2(\Sigma)}+s^{-2}\|\nabla (\rho y)\|_{L^2(Q)} \\
		\leq  \Big(C \alpha_2 s^{-3/2}+ \beta^\star C c^{3/2}\Big)  s +\Big(C   \| u_0\|_{L^2(\Omega)} +Cs^{-1}  \Big( \alpha_1 T^{1/2}|\Omega|^{1/2} +  \| \nabla u_0\|_{L^2(\Omega)}+\| u_1 \|_{L^2(\Omega)}\Big)\Big) s^{- 1/2} e^{-s}.
\end{multline}
\end{prop}
\begin{proof}
This follows from \Cref{Lemma_relation_control} with $r=0$, \Cref{Thm-regularity-one} and \Cref{estimation_lemma1-bis} with $B=-f(\widehat{y})$.
\end{proof}	

\subsubsection{Stability of the class $\widetilde\C(s)$}\label{sec-stable-Schauder-bis}

\begin{lemma}\label{Stability_class_Schauder_bis}
	Assume that $f\in \mathcal{C}^0(\mathbb{R})$ satisfies \ref{H1} with $p=3/2$ and $\beta^\star$ small enough. Then, there exists $s\geq s_0$ large enough such that
	$\Lambda_s\big(\widetilde\C(s) \big) \subset \widetilde\C(s)$.
	\end{lemma} 

\begin{proof}
For  any $s\geq s_0$ and $\yy\in \widetilde\C(s)$, let $y:=\Lambda_s(\yy)$. From the inequalities \eqref{linearized-estimate-Schauder-1-BIS}, we obtain that $\limsup_{s\to+\infty} s^{-1}\|\rho y \|_{L^2(Q)}\le \beta^\star C   c^{3/2}$, 
$\limsup_{s\to+\infty} s^{-2}\|(\rho y)_t \|_{L^2(Q)}\le \beta^\star C   c^{3/2}$
and \newline
$\limsup_{s\to+\infty} s^{-3}\|\nabla (\rho y) \|_{L^2(Q)}\le \beta^\star C   c^{3/2}$. Therefore, if $\beta^\star>0$ is small enough so that $\beta^\star C   c^{3/2} <1$, then for any $s\geq s_0$ large enough, $y\in \widetilde\C(s)$.
\end{proof}

\subsubsection{Relative compactness of the set $\Lambda_s(\widetilde\C(s))$}\label{Relative_Compactness_Schauder-bis}

\begin{lemma}\label{RelativeCompactness-bis}
Under the assumptions of Lemma \ref{Stability_class_Schauder}, $\Lambda_s(\widetilde\C(s))$ is a relatively compact subset of $\widetilde\C(s)$ for the $L^2(Q)$ norm.
\end{lemma}
\begin{proof}
Since $\Lambda_s(\widetilde\C(s))\subset \widetilde\C(s)$, this follows from the fact that $\widetilde\C(s)$ is bounded in $H^1(Q)$.
\end{proof}

\begin{remark}\label{compact}
In fact  $\widetilde\C(s)$ is a compact subset of  $L^p(Q)$ norm, $1\le  p<2+\frac{4}d$ and thus $\Lambda_s(\widetilde\C(s))$ is a relatively compact subset of $\widetilde\C(s)$ for the $L^p(Q)$ norm.
\end{remark}

\subsubsection{Continuity of the map $\Lambda_s$ in $\widetilde\C(s)$}\label{Continuity-Schauder-bis}

\begin{lemma}\label{ContinuityMap-bis}
Assume that $f\in \C^0(\mathbb{R})$ satisfies \ref{H1} with $p=3/2$. Then, the map $\Lambda_s : \widetilde\C(s)\rightarrow \widetilde\C(s)$ is continuous with respect to the $L^2(Q)$ norm.
\end{lemma}
\begin{proof}
Let $(\yy_n)_{n\in\mathbb{N}}$ be a sequence in $\widetilde\C(s)$ such that $(\yy_n)_{n\in\mathbb{N}}$ converges to $\yy\in \widetilde\C(s)$ with respect to the $L^2(Q)$ norm. Let $y_n := \Lambda_s(\yy_n)$ for all $n\in \mathbb{N}$. 
Since $\yy_n\to \yy$ in $L^2(Q)$ then $\yy_n\to \yy$ in $L^{2+\frac2d}(Q)$ (since, see Remark \ref{compact},  $\widetilde\C(s)$ is a compact subset of  $L^{2+\frac2d}(Q)$). Thus
there exist a subsequence $(\yy_{n_k})_{n_k\in\mathbb{N}}$  and $z\in L^{2+\frac2d}(Q)$ such that $\yy_{n_k}\to \yy$ a.e. and $|\yy_{n_k}|\le z$ a.e. for all $n_k$. It follows since $f\in\C^0(\R)$ that $f(\yy_{n_k})\to f(\yy)$ a.e. 
But, for all $n_k\in\N$, a.e: 
$$ \begin{aligned}
|f(\yy_{n_k})|
&\le  \alpha_1 + |\yy_{n_k}|\left( \alpha_2 + \beta^\star\ln_+^{3/2}(\yy_{n_k}) \right)\le  \alpha_1 + \alpha_2  |\yy_{n_k} |+
  \beta^\star|\yy_{n_k}|\ln_+^{3/2}(\yy_{n_k})  \\
  &\le  \alpha_1 + \alpha_2  |z|+
  \beta^\star|z|\ln_+^{3/2}(z)  \le  \alpha_1 + \alpha_2  |z|+
  C\beta^\star |z|^{1+\frac1d} 
  \end{aligned}$$
and $\alpha_1 + \alpha_2  |z|+
  C\beta^* |z|^{1+\frac1d}\in L^{2}(Q)$, since $z\in L^{2+\frac2d}(Q)$ and thus, from the dominated convergence theorem, $f(\yy_{n_k})\to f(\yy)$ in $L^2(Q)$. Indeed $f(\yy_{n })\to f(\yy)$ in $L^2(Q)$. Otherwise, there exists  $\varepsilon>0$ and a subsequence $(\yy_{n_k})_{n_k\in\mathbb{N}}$ of $(\yy_n)_{n\in\mathbb{N}}$ such that $\|f(\yy_{n_k})-f(\yy)\|_{L^2(Q)}\ge \varepsilon$.
  But, arguing as before, there exists a subsequence $(\yy_{n_{k'}})_{n_{k'}\in\mathbb{N}}$  of  $(\yy_{n_k})_{n_k\in\mathbb{N}}$ such that $f(\yy_{n_{k'}})\to f(\yy)$ in $L^2(Q)$ which leads to a contradiction.

  In particular, $z_n := y_n - y = \Lambda_s(\yy_n) - \Lambda_s(\yy)$ satisfies $z_n = \rho^{-2} L(w_n - w)$ with $w_n - w$ solution of \eqref{bilinear equation} associated with data $\left(0,0, f(\widehat{y})-f(\widehat{y}_n)\right)$. Thus, using that $L^2(0,T;L^{\frac2{1+ \varepsilon}}(\Omega)) \hookrightarrow L^2(0,T;H^{-1}(\Omega))$, estimate \eqref{weighted_estimate_control_traj-frac} with $B=f(\yy_n)-f(\yy )$, $u_0=u_1=0$ implies

We denote by $v_n$ and $v$ the associated control to $y_n$ and $y$ respectively and we have, by definition of the operator $\Lambda_s$, that $(y_n, v_n) = \rho^{-2}(Lw_n, s\eta^2\Psi\partial_\nu w_n)$ and  $(y, v) = \rho^{-2}\left( Lw, s\eta^2\Psi\partial_\nu w\right)$ with $w_n$ and $w$ solution of \eqref{bilinear equation} associated with $(u_0, u_1, -f(\widehat{y}_n))$ and $(u_0, u_1, -f(\widehat{y}))$ respectively. In particular, $z := y_n - y = \Lambda_s(\yy_n) - \Lambda_s(\yy)$ satisfies $z = \rho^{-2} L(w_n - w)$  with $w_n - w$ solution of \eqref{bilinear equation} associated with data $\left(0,0, f(\widehat{y})-f(\widehat{y}_n)\right)$. Then estimate \eqref{weighted_estimate_control_traj-frac} with $B=f(\yy_n)-f(\yy )$, $u_0=u_1=0$ and  $r=0$ implies $\|\rho (y_n-y)\|_{L^2(Q)} \leq Cs^{-3/2}\|\rho(f(\yy)-f(\yy_n))\|_{L^2(Q)}$ so that  $y_n \rightarrow y$ as $n\rightarrow +\infty$ in $L^2(Q)$.
\end{proof}

\appendix

\section{Appendix: Proof of \Cref{Thm-regularity-one} }\label{sec:appendix}

We check that the optimal pair $(y,v)$ defined in \Cref{Lemma_relation_control} belongs to $\C^0([0,T];H^1(\Omega))\times \C^0(0,T; H^{1/2}(\partial\Omega))$ as soon as $B\in L^2(Q)$ and $(u_0,u_1)\in\boldsymbol{V}$. This property has been proved in \cite{Bhandari_MCSS_2022} in the one dimensional case
(by generalizing \cite{ervedoza_zuazua_2010}). The occurence of the weights and the coupling between the primal variable $y$ and the dual one $w$ (see Remark \ref{coupling_ys_ws}) make the arguments and the computations more involved with respect \cite{ervedoza_zuazua_2010}. The proof is divided into four steps: the first and second steps are similar to \cite[Appendix]{Bhandari_MCSS_2022}; the third one is new with respect to \cite[Appendix]{Bhandari_MCSS_2022} as it provides an estimate of $\Vert \nabla(\rho y)\Vert_{L^2(Q)}$.

For all $f\in \C^0(\mathbb R;E)$  (where $E$ is a Banach space) and any $\tau\neq 0$, we define $\delta_\tau f:= f\left(t+\frac{\tau}{2} \right)- f\left(t-\frac{\tau}{2} \right)$, 
$\mathcal{T}_\tau f := \frac{1}{\tau}\delta_\tau \left(\frac{\delta_\tau f}{\tau}\right)$ and $\widetilde\delta_\tau f(t):=\frac{f(t+\tau)-f(t)}{\tau}$.

Let now $w\in P_s$ and $y \in L^2(Q)$  be given by \Cref{Lemma_relation_control}.  Then $\mathcal{T}_\tau w$ belongs to $P_s$, where $w$ as well as $y$  can be extended uniquely   on $(-\infty,0)$ and $(T,+\infty)$. Indeed, in the time interval $(-\infty, 0)$ the solution $y$ satisfies
$$
L y = 0 \text{ in } \Omega\times (-\infty, 0), \quad  y=0\text{ on } \partial\Omega \times (-\infty, 0), \quad (y(\cdot,0),\partial_t y(\cdot,0)) =(u_0,u_1) \text{ in } \Omega ,
$$
where the source term $B\in L^2(Q)$ is assumed to be extendable by $0$ outside $(0,T)$. Recall that the boundary condition $y(\cdot,t) =0$ on $\Gamma_0$ holds outside $(0,T)$ since $\eta=0$ (appearing in the formula of $v$) vanishes outside $(\delta, T-\delta)$, see \eqref{function_eta}.

Similarly, in $(T, + \infty)$ we can define the solution $y$ uniquely, and $y(t)=0$ for all $t\ge T$. It follows that the solution $y$ satisfies $y\in\C^0(\mathbb R;L^2(\Omega))\cap \C^1(\mathbb R;H^{-1}(\Omega))$ and  $y\in \C^0((-\infty,\delta ];H_0^1(\Omega))\cap \C^1((-\infty,\delta];L^2(\Omega))$ and  $y\in\C^0([T-\delta, +\infty  );H_0^1(\Omega))\cap \C^1([T-\delta,+ \infty);L^2(\Omega))$ (see \cite[Theorem 2.1, page 151]{Lasi-Trig-Lions}).
We extend as well the weight $\rho$ in $\Omega\times \mathbb{R}$. This ensures the extension of the solution $w$ which satisfies the following set of equations in $\mathbb R$
\begin{equation}\label{equation-w}
		L w = \rho^2 y  \text{ in } \Omega\times \mathbb{R}, \qquad w=0 \text{ on } \partial\Omega\times \mathbb{R}.
	\end{equation}  
 Moreover, it can be seen that $Lw=0$ in $[T, + \infty)$. 

\par\noindent
{\bf Step 1.} We assume that $u_0\in H^2(\Omega)\cap H^1_0(\Omega)$, $u_1\in H_0^1(\Omega)$ and $B\in \mathcal{D}(0,T;L^2(\Omega))$ and prove that $v\in H^1(0,T;L^2(\partial\Omega))$ and $y_t\in L^2(Q)$.

Since $w \in \CC^0(\mathbb R;H^1_0(\Omega))\cap \CC^1(\mathbb R;L^2(\Omega))$  solves \eqref{equation-w}, $\mathcal{T}_\tau w \in \CC^0([0,T]; H^1_0(\Omega)) \cap \CC^1([0,T]; L^2(\Omega))$ and $\partial_\nu \mathcal{T}_\tau w\in L^2(\Sigma )$. With  $z=\mathcal{T}_\tau w$ as test function in \eqref{bilinear equation}, we have  
\begin{multline}\label{Formulation_higher}
	\intq \rho^{-2}  L w L \mathcal{T}_\tau w \dx\dt + s  \int_{\Sigma}  \eta^2(t) \Psi(x)\rho^{-2} \partial_\nu w   \mathcal{T}_\tau \partial_\nu  w  d\Sigma \\
	=   \int_\Omega u_1  \mathcal{T}_\tau w(\cdot,0) \dx   - \int_\Omega u_0 \mathcal{T}_\tau w_t(\cdot,0) \dx + \intq B  \mathcal{T}_\tau w \dx\dt.
\end{multline}
 Proceeding as \cite[Appendix]{Bhandari_MCSS_2022} and using that the smooth function $\eta$ given by \eqref{function_eta} satisfies $\eta=0$ in $(-\infty, \delta]\cup [T-\delta, +\infty)$ (with $\delta>0$ given in \eqref{function_eta}), we get the following estimate (we refer to \cite[sub-steps 1 and 2 pages 108-110]{Bhandari_MCSS_2022})
\begin{align}\label{3} 
&\int_{Q} \rho^2(t) \left|\widetilde\delta_\tau(\rho^{-2} Lw )(t)\right|^2 \dx\dt   +s \int_{\Sigma} \Psi  \frac{ \eta^2(t)\rho^{-2}(t)+\eta^2(t+\tau)\rho^{-2}(t+\tau)  }{2}\left|\widetilde\delta_\tau(\partial_\nu  w)(t)\right|^2 d \Sigma 
\notag \\
= &    \int_{Q} \rho^2(t)\;\widetilde\delta_\tau(\rho^{-2} Lw)(t)\;\widetilde\delta_\tau(\rho^{-2})(t)  Lw(t+\tau) \dx\dt  
 -\frac1\tau \int_{-\tau}^{0}\int_\Omega \rho^{-2}(t+\tau)Lw(t+\tau) \;\widetilde\delta_\tau Lw(t)  \dx\dt 
\notag \\
&- \frac{1}{\tau}\int_T^{T-\tau}\int_\Omega \rho^{-2}(t+\tau)Lw(t+\tau) \;\widetilde\delta_\tau Lw(t)  \dx\dt  \notag \\
& -  s  \int_{\Sigma}  \Psi\;\widetilde\delta_\tau (\eta^2\rho^{-2})(t)  \frac{\partial_\nu w (t)+\partial_\nu w  (t+\tau) }{2} \;\widetilde\delta_\tau (\partial_\nu w ) (t)d\Sigma
\notag  \\
& - \intq B \, \mathcal{T}_\tau w(t) \dx\dt  -\int_\Omega u_1  \mathcal{T}_\tau w(\cdot,0) \dx  + \int_\Omega u_0 \mathcal{T}_\tau w_t(\cdot,0) \dx   
\end{align}
Then, following \cite[step 1, sub-step 3 page 111-114]{Bhandari_MCSS_2022}, we get that each term of the right hand side of \eqref{3} are uniformly bounded with respect to $\tau\in [0,\delta]$. Then we can conclude, from \eqref{3}, that the terms $\int_{Q} \rho^2(t) \left|\widetilde\delta_\tau(\rho^{-2} Lw )(t)\right|^2 \dx\dt$
and
$$ \int_{\Sigma} \Psi  \frac{ \eta^2(t)\rho^{-2}(t)+\eta^2(t+\tau)\rho^{-2}(t+\tau)  }{2}\left|\widetilde\delta_\tau(\partial_\nu  w)(t)\right|^2 d \Sigma $$
are bounded.  {\color{black}
Thus $ \rho  (\rho^{-2} Lw )_t\in L^2(Q)$, $\Psi^{1/2}  \eta  \rho^{-1}  \partial_\nu  w_t\in L^2(\Sigma)$, and $y\in H^1(0,T;L^2(\Omega))$, $v\in H^1(0,T; L^2(\partial\Omega))$.

Eventually, following \cite[step 1, sub-step 4 page 115]{Bhandari_MCSS_2022}, we prove that  $v\in \CC^0([0,T];H^{1/2}(\partial\Omega))$ and $y\in \CC^0([0,T];H^1(\Omega))\cap \CC^1([0,T];L^2(\Omega))$.

Since $\widetilde\delta_\tau w\in P$,  Carleman estimates \eqref{Carleman-in-0-T} gives
	    	\begin{multline}\label{ttt}
	    		s \intq \rho^{-2}  \biggl(|\widetilde\delta_\tau w_t|^2 + |\widetilde\delta_\tau \nabla w|^2\biggr) \dx \dt + s^3 \intq  \rho^{-2}  |\widetilde\delta_\tau w|^2 \dx \dt \\
	    		+	s \int_\Omega \rho^{-2}(0) \biggl(|\widetilde\delta_\tau(w_t)(0)|^2 + |\widetilde\delta_\tau(\nabla w)(0)|^2\biggr) \dx + s^3 \int_\Omega  \rho^{-2}(0) |\widetilde\delta_\tau w(0)|^2 \dx  \\
	    		\leq C \intq \rho^{-2}  |L (\widetilde\delta_\tau w)|^2 \dx \dt  + C s \int_{\Sigma} \eta^2 \Psi\rho^{-2} |\widetilde\delta_\tau(\partial_\nu  w)|^2 d \Sigma. 
	    	\end{multline}
Remark that  
  $\int_{\Sigma} \Psi\eta^2 \rho^{-2} \left|\widetilde\delta_\tau(\partial_\nu  w)\right|^2 d \Sigma 
$ and 
$\intq \rho^{-2} |L (\widetilde\delta_\tau w)|^2 \dx \dt$ are bounded since 
$$\int_{\Sigma} \Psi\eta^2(t) \rho^{-2}(t) \left|\widetilde\delta_\tau(\partial_\nu  w)(t)\right|^2 d \Sigma
\le 2\int_{\Sigma} \Psi\frac{\eta^2(t)\rho^{-2}(t)+\eta^2(t+\tau)\rho^{-2}(t+\tau)}{2} \left|\widetilde\delta_\tau(\partial_\nu  w)(t)\right|^2 d \Sigma 
$$
and
$$
 \intq \rho^{-2}(t) |L (\widetilde\delta_\tau w)(t)|^2 \dx \dt
\le 2 \int_{Q} \rho^2(t) \left|\widetilde\delta_\tau(\rho^{-2} Lw  )(t)\right|^2 \dx\dt\
+2\int_{Q} \rho^2(t)\left|\widetilde\delta_\tau(\rho^{-2})(t)Lw (t+\tau)\right|^2 \dx\dt
$$
thus the right hand side in \eqref{ttt} is bounded. Therefore  $\tau \mapsto \widetilde\delta_\tau w_t$ and $\tau \mapsto \widetilde\delta_\tau \nabla w$ are bounded in $L^2(Q)$ and thus $w_{tt}\in L^2(Q)$ and $w_{t}\in L^2(0,T;H^1_0(\Omega))$. Moreover, $\tau \mapsto \widetilde\delta_\tau w(0)$ is bounded in $H^1_0(\Omega)$ thus $w_t(0)\in H^1_0(\Omega)$. Eventually, since $\tau \mapsto L (\widetilde\delta_\tau w) $ is bounded in $L^2(Q)$, we conclude that $Lw_t\in L^2(Q)$ and then that $w_t\in P$; it follows that $\partial_\nu  w_{t } \in L^2(\Sigma)$.

 \begin{remark} We then have  $w \in \CC^1([0,T];H^1_0(\Omega))\cap \CC^2([0,T];L^2(\Omega))$  and since $Lw \in \CC^0([0,T];L^2(\Omega))$  we deduce that $\Delta w \,(=w_{tt}-L w)  \in  \CC^0([0,T];L^2(\Omega))$ and thus  $w \in \CC^0 ([0,T];H^2(\Omega))$. 
 \end{remark}
 \par\noindent
 Since $w \in \CC^0 ([0,T];H^2(\Omega))$ $v \in  \CC^0 ([0;T];H^{1/2} (\partial\Omega))\cap H^1(0,T;L^2(\partial\Omega))$ and $v$ satisfies the compatibility conditions $v(0)=0$, \cite[Theorem 2.1, page 151]{Lasi-Trig-Lions} leads to   
$y\in  \CC^0 ([0;T];H^1 (\Omega))\cap \CC^1 ([0;T];L^2(\Omega))\cap \CC^2 ([0;T];H^{-1}(\Omega))$.

\begin{remark} \label{remark9}
 As previously mentioned, $y$ and $w$ can be extended uniquely on $(-\infty;0)$ and $(T,+\infty)$ so that $w \in \CC^0 (\R;H^2(\Omega))\cap\CC^1(\R;H^1_0(\Omega))\cap \CC^2(\mathbb{R};L^2(\Omega))$,  $\partial_\nu  w  \in \CC^0(\R; H^{1/2}(\partial\Omega))$,  $\partial_\nu  w_{t } \in L^2_{\rm loc}(\R;L^2(\partial\Omega))$ 
  and $y \in \CC^0(\R;H^1(\Omega))\cap \CC^1(\R;L^2(\Omega))$.
 \end{remark}

 \par\noindent
{\bf Step 2.} We prove estimate on $v_t$ and $y_t$ in \eqref{estimate_regular}.
 }
Carleman estimate \eqref{Carleman-in-0-T} for $w_t\in P$ reads
	    	\begin{multline}\label{Carleman-w-t}
	    		s \intq \rho^{-2}  (| w_{tt}|^2 + | \nabla w_t|^2) \dx \dt + s^3 \intq  \rho^{-2}  | w_t|^2 \dx \dt \\
	    		+	s \int_\Omega \rho^{-2}(0) (|w_{tt}(0)|^2 + \nabla w_t(0)|^2) \dx + s^3 \int_\Omega  \rho^{-2}(0) | w_t(0)|^2 \dx  \\
	    		\leq C \intq \rho^{-2} |L w_t|^2 \dx \dt  + C s \int_{\Sigma} \eta^2\Psi  \rho^{-2}  |\partial_\nu  w_t|^2 d \Sigma.
	    	\end{multline}
\par\noindent		
Proceeding as in \cite[step 2 sub-step 1, page 116]{Bhandari_MCSS_2022}, we pass to the limit when $\tau\to 0$ in the equality \eqref{3} and obtain 
\begin{equation}\label{estimation-1-y'}
\begin{aligned}
\int_{Q} &\rho^2\left|y_t\right|^2 \dx\dt +s\int_{\Sigma}\Psi \eta^2 \rho^{-2} \left|\partial_\nu w_t\right|^2 d\Sigma
=-4s\lambda \beta\int_{Q}\big(t-\frac T2\big) \phi \rho^2  y_ty \dx\dt
- \int_\Omega y (0)( \rho^2y )_t(0) \dx \\
&\hspace{1cm}-s\int_{\Sigma}\Psi(\eta^2\rho^{-2})_t \partial_\nu w\partial_\nu w_t d\Sigma -\intq B  w_{tt} \dx\dt-\int_\Omega  w_{tt}(0)u_1\dx\\
&\hspace{1cm}-2s\lambda \beta T\int_\Omega  \phi(0)\rho^2(0)u_0^2\dx+ \int_\Omega \rho^2(0) u_1 u_0\dx-\int_\Omega \nabla u_0\nabla w_t(0) \dx
\end{aligned}
\end{equation}
and then estimate each term of the right side. In particular, for the third term, we write, using $(\rho^{-1})_t=-2s\lambda \beta (t-\frac T2) \phi\rho^{-1}$ that 
\begin{align*}
\big|s\int_{\Sigma}\Psi(\eta^2\rho^{-2})_t \partial_\nu w\partial_\nu w_t d\Sigma|&\le 2\left(s\int_{\Sigma}\Psi |(\eta\rho^{-1})_t|^2|\partial_\nu w|^2d\Sigma\right)^{1/2}\left(s\int_{\Sigma}\Psi \eta^2(t)\rho^{-2}(t)|\partial_\nu w_t|^2 d\Sigma\right)^{1/2}\\
&\le C\left(s^3\int_{\Sigma}\Psi\eta^2 \rho^{-2} |\partial_\nu w|^2 d \Sigma +s\int_{\Sigma}\Psi\rho^{-2} |\partial_\nu w|^2 d \Sigma \right)^{1/2}
\left(s\int_{\Sigma}\Psi \eta^2 \rho^{-2} | \partial_\nu w_t|^2d \Sigma\right)^{1/2}\\
&\le C\left( s\int_{\Sigma} \frac{\rho^2 }{\eta^2\Psi }v ^2 d\Sigma+s\int_{\Sigma}\Psi\rho^{-2} |\partial_\nu w|^2 d\Sigma\right) +\frac s8\int_{\Sigma}\Psi \eta^2 \rho^{-2} | \partial_\nu w_t|^2d \Sigma.
\end{align*}
We now estimate the term $\int_{\Sigma}\Psi\rho^{-2} |\partial_\nu w|^2 d\Sigma$ appearing in the previous inequality: proceeding as in \cite[Lemma 3.7]{Lions-1} with $q(x,t)=h\rho^{-2}(x,t)$
where $h\in\big(\CC^1(\overline{\Omega})\big)^d$ satisfies $h(x)=\nu(x)$ on $\partial\Omega$, we get (with the notation $\displaystyle f_kg_k=\sum_{k=1}^d f_kg_k$) 
$$\begin{aligned}
\frac12\int_{\Sigma}\rho^{-2} |\partial_\nu w|^2 d\Sigma=
&\int_{Q}\rho^{-2}\nabla w\cdot\nabla h_k\frac{\partial w}{\partial x_k}+2\int_{Q}\rho^{-1}\left(\nabla w\cdot\nabla\rho^{-1} - (\rho^{-1})_t w_t\right)h\cdot\nabla w\\
&+\frac12\int_{Q}\rho^{-2} \big(|w_t|^2-|\nabla w|^2\big)\nabla\cdot h+\int_{Q}\rho^{-1} \big(|w_t|^2-|\nabla w|^2\big)h\cdot\nabla\rho^{-1}\\
&+\int_\Omega \left[\rho^{-2}w_th\cdot\nabla w\right]_0^T-\int_{Q} yh\cdot\nabla w.
\end{aligned}
$$
Writing that $\vert\nabla \rho^{-1}\vert \leq Cs \rho^{-1}$ and $|(\rho^{-1})_t|\leq Cs\rho^{-1}$, and using that  $h\in\big(\CC^1(\overline{\Omega})\big)^d$  and $s\ge 1$ we obtain 
$$\begin{aligned}
\int_{\Sigma}\rho^{-2} |\partial_\nu w|^2 d\Sigma
\leq 
&Cs\int_\Omega \biggl(\rho^{-2}(0)(|w_t|^2+|\nabla w|^2)(0)+\rho^{-2}(T)(|w_t|^2+|\nabla w|^2)(T)\biggr)\d x \\
&+Cs\int_{Q} \rho^{-2}(|w_t|^2+|\nabla w|^2)\d x\d t +Cs^{-1}\int_{Q}\rho^2|y|^2 \d x\d t
\end{aligned}$$
leading, using \eqref{Carleman-in-0-T} and that $\Psi\in\big(\CC^1(\overline{\Omega})\big)^d$ to
\begin{equation}\label{estim-w-x}
s \int_{\Sigma}\Psi \rho^{-2} |\partial_\nu w|^2 d\Sigma\le Cs\left(\|\rho y \|_{L^2(Q)}^2+s^{-1}\biggl\|\frac{\rho}{\eta\Psi^{1/2}}v\biggr\|_{L^2(\Sigma)}^2\right).
\end{equation}
Thus, 
$$
|s\int_{\Sigma}\Psi(\eta^2\rho^{-2})_t \partial_\nu w\partial_\nu w_t d\Sigma|\le Cs\left(\|\rho y \|_{L^2(Q)}^2+\biggl\|\frac{\rho}{\eta\Psi^{1/2}}v \biggr\|_{L^2(\Sigma)}^2\right)+\frac s8\int_{\Sigma}\Psi \eta^2 \rho^{-2} |\partial_\nu w_t|^2 d\Sigma.
$$
Proceeding as in \cite[step 2 sub-step 2, page 117-120]{Bhandari_MCSS_2022} for the other terms in the right hand side of \eqref{estimation-1-y'}, we then get from \eqref{estimation-1-y'}
\begin{equation}\label{estimation-2-y'}
\begin{aligned}
\int_{Q}  \rho^2\left|y_t\right|^2 \dx\dt &+s \int_{\Sigma} \Psi\eta^2 \rho^{-2} |\partial_\nu w_t|^2  d\Sigma\\
&\leq C\left(s^2\|\rho y \|^2_{L^2(Q)}+s \|\frac{\rho}{\eta\Psi^{1/2}}v\|_{L^2(\Sigma)}^2 + s^{-1}\Vert \rho B\Vert^2_{L^2(Q)} \right. \\
&\qquad \left.+s\|\rho(0)u_0\|^2_{L^2(\Omega)} +s^{-1}\|\rho(0)\nabla u_0 \|_{L^2(\Omega)} ^2+s^{-1}\|\rho(0) u_1\|_{L^2(\Omega)}^2\right).
\end{aligned}
\end{equation}
We have, since $\eta\in \CC^1([0,T])$ and $v = s\eta^2 \Psi \rho^{-2}\partial_\nu w$
\begin{equation}\label{estimation-v-t}
 s^{-1}\int_{\Sigma} \frac{\rho^2}{\eta^2\Psi }   \left|v_t\right|^2d\Sigma 
 \le Cs\left(  \int_{\Sigma} \Psi \eta^2 \rho^{-2} |\partial_\nu w_t|^2 d\Sigma+ \int_{\Sigma} \frac{\rho^2}{\eta^2\Psi}v^2 d\Sigma
+ \int_{\Sigma} \Psi  \rho^{-2} |\partial_\nu w|^2 d\Sigma\right)
\end{equation}
thus, using that $\Psi\in\big(\CC^1(\overline{\Omega})\big)^d$ and \eqref{estim-w-x}, \eqref{estimation-2-y'}, implies for $s\ge s_0\ge 1$ that
$$
\begin{aligned}
\int_{Q} &\rho^2 \left|y_t\right|^2 \dx\dt +s^{-1}\int_{\Sigma} \frac{\rho^2}{\eta^2\Psi}   \left|v_t\right|^2d\Sigma
\\
& \le C \left(s^{-1} \|\rho B\|^2_{L^2(Q)} +s\|\rho(0)u_0\|^2_{L^2(\Omega)} +s^{-1}\|\rho(0)\nabla u_0\|_{L^2(\Omega)} ^2+s^{-1}\|\rho(0) u_1\|_{L^2(\Omega)}^2\right),
\end{aligned}
$$
{\color{black} which is the announced estimates \eqref{estimate_regular} for the first two terms in the case of regular data.

\par\noindent
{\bf Step 3.} We obtain the estimate on $\Vert \nabla(\rho y)\Vert_{L^2(Q)}$  and $\Vert \rho v\Vert_{L^\infty(0,T; H^{1/2}(\partial\Omega))}$ in  \eqref{estimate_regular}. With respect to \cite{Bhandari_MCSS_2022}, this part is new. From the definition of $v$, since $Lw(T)=(\rho^2 y)(T)=0$ and $s\ge 1$ we have 
\begin{equation}
\label{estimate_V-L2h1/2}
\begin{aligned}
\|\rho v\|_{L^2(0,T;H^{1/2}(\partial\Omega))}
&={ s}\|\Psi\rho^{-1}\partial_\nu w\|_{L^2(0,T;H^{1/2}(\partial\Omega))}={ s}\|\Psi\rho^{-1}h\cdot\nabla w\|_{L^2(0,T;H^{1/2}(\partial\Omega))}\\
&\le Cs\|\Psi \rho^{-1}h\cdot \nabla w  \|_{L^2(0,T;H^{1}(\Omega))}  \\
&\le C(s\|\nabla( \rho^{-1} w)  \|_{L^2(0,T;H^{1}(\Omega))}+s^2\|  \rho^{-1} w  \|_{L^2(0,T;H^{1}(\Omega))})\\
\end{aligned}
\end{equation}
and
\begin{equation}
\label{estimate_V}
\begin{aligned}
\|\rho v\|_{L^\infty(0,T;H^{1/2}(\partial\Omega))}
&={ s}\|\Psi\rho^{-1}\partial_\nu w\|_{L^\infty(0,T;H^{1/2}(\partial\Omega))}={ s}\|\Psi\rho^{-1}h\cdot\nabla w\|_{L^\infty(0,T;H^{1/2}(\partial\Omega))}\\
&\le Cs\|\Psi \rho^{-1}h\cdot \nabla w  \|_{L^\infty(0,T;H^{1}(\Omega))}  \\
&\le C(s\|\nabla( \rho^{-1} w)  \|_{L^\infty(0,T;H^{1}(\Omega))}+s^2\|  \rho^{-1} w  \|_{L^\infty(0,T;H^{1}(\Omega))}).
\end{aligned}
\end{equation}
Since $\nabla(\rho^{-1}w)=\nabla(\rho^{-1})w+\rho^{-1} \nabla w$ and
$\Delta (\rho^{-1}w)=\Delta (\rho^{-1})w+2\nabla(\rho^{-1}) \cdot \nabla w+ \rho^{-1} \Delta w$ we have
$$\begin{aligned}
\|  \rho^{-1} w  \|_{L^2(0,T;H^{1}(\Omega))}&\le{ C} \|\nabla( \rho^{-1} w)  \|_{L^2(0,T;L^2(\Omega))}\\
&\le C\Big(s\|\rho^{-1} w\|_{L^2(0,T;L^2(\Omega))}+\|\rho^{-1} \nabla w\|_{L^2(0,T;L^2(\Omega))}\Big),
\end{aligned}$$
$$\begin{aligned}
\|\nabla( \rho^{-1} w)  \|_{L^2(0,T;H^{1}(\Omega))}
\le &C\|\Delta (\rho^{-1} w)   \|_{L^2(0,T;L^2(\Omega))}\\
\le &C\Big(s^2\|\rho^{-1} w\|_{L^2(0,T;L^2(\Omega))}+s\|\rho^{-1} \nabla w\|_{L^2(0,T;L^2(\Omega))}+\| \rho^{-1} \Delta w\|_{L^2(0,T;L^2(\Omega))}\Big)\\
\le &C\Big(s^2\|\rho^{-1} w\|_{L^2(0,T;L^2(\Omega))}+s\|\rho^{-1} \nabla w\|_{L^2(0,T;L^2(\Omega))}+\| \rho^{-1}  w_{tt}\|_{L^2(0,T;L^2(\Omega))}\\
&+\| \rho  y\|_{L^2(0,T;L^2(\Omega))}\Big)
\end{aligned}$$
and thus 
$$
\begin{aligned}
\|\rho v\|_{L^2(0,T;H^{1/2}(\partial\Omega))}\le 
&C\Big(s^3\|\rho^{-1} w\|_{L^2(0,T;L^2(\Omega))}+s^2\|\rho^{-1} \nabla w\|_{L^2(0,T;L^2(\Omega))}+s\| \rho^{-1} w_{tt}\|_{L^2(0,T;L^2(\Omega))}\\
&+s\| \rho  y\|_{L^2(0,T;L^2(\Omega))} \Big).
\end{aligned}
$$
Estimate \eqref{Carleman-in-0-T} then rewrites :
\begin{equation}\label{Carleman-fin-w}
\begin{aligned}  
s^{1/2}\|  \rho^{-1}(0) w_{t}(0)&\|_{L^2(\Omega)}+s^{1/2}\|  \rho^{-1}(0)  \nabla w(0)\|_{L^2(\Omega)} +s^{3/2}\|  \rho^{-1}(0)   w(0)\|_{L^2(\Omega)} \\
 +s^{1/2}\|  \rho^{-1} &w_{t}\|_{L^2(Q)}+ s^{1/2}\|  \rho^{-1}\nabla w \|_{L^2(Q))}+s^{3/2}\|  \rho^{-1} w\|_{L^2(Q)} \\
\le & C \Big(\|\rho y\|_{L^2(Q)}+s^{-1/2}\biggl\|\frac{\rho}{\eta\Psi^{1/2}}v\biggr\|_{L^2(\Sigma)}\Big)\\
\le & C\Big( s^{-3/2} \|\rho B\|_{L^2(Q)} +s^{-1/2} \| \rho(0)u_0\|_{L^2(\Omega)} +s^{-3/2} \| \rho(0) u_1\|_{L^2(\Omega)}\Big)
\end{aligned} 
\end{equation}
while estimate \eqref{Carleman-w-t}  rewrites (using  estimates \eqref{estimation-2-y'}-\eqref{estimation-v-t})
\begin{equation}\nonumber
\begin{aligned}  
s^{1/2}\|  \rho^{-1}(0) &w_{tt}(0)\|_{L^2(\Omega)}+s^{1/2}\|  \rho^{-1}(0)  \nabla w_t(0)\|_{L^2(\Omega)} +s^{3/2}\|  \rho^{-1}(0)   w_t(0)\|_{L^2(\Omega)} \\
 +s^{1/2}\|  \rho^{-1} &w_{tt}\|_{L^2(Q)}+ s^{1/2}\|  \rho^{-1}\nabla w_t \|_{L^2(Q)}+s^{3/2}\|  \rho^{-1} w_t\|_{L^2(Q)} \\
\le & C \Big(s\|\rho y\|_{L^2(Q)}+\|\rho y_t\|_{L^2(Q)}+s^{1/2}\biggl\|\frac{\rho}{\eta\Psi^{1/2}}v\biggr\|_{L^2(\Sigma)}+s^{-1/2}\biggl\|\frac{\rho}{\eta\Psi^{1/2}}v_t\biggr\|_{L^2(\Sigma)}\Big)\\
\le & Cs^{-1/2} \mathcal{A}(u_0,u_1,B).
\end{aligned} 
\end{equation}
with 
$$
\mathcal{A}(u_0,u_1,B):=\|\rho B\|_{L^2(Q)} + s\| \rho(0)u_0\|_{L^2(\Omega)} + \| \rho(0) u_1\|_{L^2(\Omega)}+ \| \rho(0) \nabla u_0\|_{L^2(\Omega)}.
$$
Since $s\ge 1$, $|(\rho^{-1})_t|= Cs|\rho^{-1}|$, $|(\rho^{-1})_{tt}|= Cs^2|\rho^{-1}|$, $|\nabla (\rho^{-1})|= Cs|\rho^{-1}|$, $|\Delta (\rho^{-1})|= Cs^2|\rho^{-1}|$ and $\rho^{-1}(0)\Delta w(0)=\rho^{-1}w_{tt}(0)-\rho(0)u_0$, we deduce  
$$
\begin{aligned}
\|\Delta &(\rho^{-1} w)(0)\|_{L^2(\Omega)}+\|\nabla (\rho^{-1} w)_t)(0)\|_{L^2(\Omega)}\\
&\le C\Big(\|\rho^{-1}(0) w_{tt}(0)\|_{L^2(\Omega)}+\|\rho(0) u_0\|_{L^2(\Omega)}+s\| \rho^{-1} (0)\nabla w(0)\|_{L^2(\Omega)} +s^2\|\rho^{-1}(0)w(0)\|_{L^2(\Omega)}\\
&\hspace{1cm}+s\|\rho^{-1}(0) w_t(0)\|_{L^2(\Omega)}+\|\rho^{-1}(0) \nabla w_t(0)\|_{L^2(\Omega)}\Big)\\
&\le Cs^{-1} \mathcal{A}(u_0,u_1,B)
\end{aligned}
$$
and thus 
$$
\|\rho v\|_{L^2(0,T;H^{1/2}(\partial\Omega))}\le C\mathcal{A}(u_0,u_1,B).
$$
Since  $w_{tt}-\Delta w=L  w$, $w_{|\partial\Omega}=0$, we infer, using $\rho^{-1}\in\mathcal{C}^\infty(\overline{Q})$ and $y=\rho^{-2}Lw$, that
 \begin{align}\nonumber
	\begin{dcases}
		(\rho^{-1}w)_{tt}-\Delta (\rho^{-1}w)=\widetilde B_w:=\rho y+2\rho^{-1}_tw_t+\rho^{-1}_{tt}w-2\nabla \rho^{-1}\cdot\nabla w-\Delta\rho^{-1}w, \\
		(\rho^{-1}w)_{|\partial\Omega}=0. \\
 	\end{dcases}
\end{align}  
for which the standard estimates 
$$
\|(\rho^{-1}w)_t(t)\|_{L^2(\Omega)}+\|\nabla(\rho^{-1}w)(t)\|_{L^2(\Omega)} \le C\Big(\|\widetilde B_w\|_{L^2(Q)}+\|\rho^{-1}(0)w(0)\|_{L^2(\Omega)}+\|(\rho^{-1} w)_t(0)\|_{L^2(\Omega)}\Big)
$$
and
$$
\begin{aligned}
 \|(\rho^{-1}w)_{tt}(t)\|_{L^2(\Omega)}  &+\|\nabla(\rho^{-1}w)_t(t)\|_{L^2(\Omega)}  +\|\Delta( \rho^{-1}w)(t)\|_{L^2(\Omega)} \\
 &\le C\Big(\|\widetilde B_w\|_{H^1(0,T;L^2(\Omega))}+\|\Delta (\rho^{-1} w)(0)\|_{L^2(\Omega)}+\|\nabla (\rho^{-1} w)_t (0)\|_{L^2(\Omega)}\Big)
 \end{aligned}
 $$
 hold true, for all $t\geq 0$.

Using again that $s\ge 1$, $|(\rho^{-1})_t|= Cs|\rho^{-1}|$, $|(\rho^{-1})_{tt}|= Cs^2|\rho^{-1}|$,  $|\nabla (\rho^{-1})|= Cs|\rho^{-1}|$ and  $|\Delta (\rho^{-1})|= Cs^2|\rho^{-1}|$  we have the following estimates :
$$\begin{aligned}
\|\widetilde B_w\|_{L^2(Q)}
&\le C\Big(\|\rho y\|_{L^2(Q)}+s\|\rho^{-1} w_t\|_{L^2(Q)}+s^2\|\rho^{-1} w\|_{L^2(Q)}+s\|\rho^{-1} \nabla w\|_{L^2(Q)}\Big)\\
\le & C\Big( s^{-1} \|\rho B\|_{L^2(Q)} + \| \rho(0)u_0\|_{L^2(\Omega)} +s^{-1} \| \rho(0) u_1\|_{L^2(\Omega)}\Big),
\end{aligned}
$$
$$
\nonumber
\begin{aligned}
\|\rho^{-1}(0)w(0)\|_{L^2(\Omega)}+\|(\rho^{-1} w)_t(0)\|_{L^2(\Omega)}
&\le C\Big(s\|\rho^{-1}(0)w(0)\|_{L^2(\Omega)}+\|\rho^{-1} w_t(0)\|_{L^2(\Omega)}\Big)\\
\le & C\Big( s^{-2} \|\rho B\|_{L^2(Q)} +s^{-1} \| \rho(0)u_0\|_{L^2(\Omega)} +s^{-2} \| \rho(0) u_1\|_{L^2(\Omega)}\Big)
\end{aligned}
$$
and therefore
\begin{equation}\nonumber 
\|(\rho^{-1}w)_t(t)\|_{L^2(\Omega)}+\|\nabla(\rho^{-1}w)(t)\|_{L^2(\Omega)} \le   C\Big( s^{-1} \|\rho B\|_{L^2(Q)} + \| \rho(0)u_0\|_{L^2(\Omega)} +s^{-1} \| \rho(0) u_1\|_{L^2(\Omega)}\Big).
\end{equation}
We also have
$$
(\widetilde B_w)_t
=\rho_t y+\rho y_t  +3\rho^{-1}_{tt}w_t+2\rho^{-1}_tw_{tt}+\rho^{-1}_{ttt}w-2\nabla \rho^{-1}_t\cdot\nabla w-2\nabla \rho^{-1}\cdot\nabla w_t-\Delta\rho^{-1}_tw-\Delta\rho^{-1}w_t$$
and thus
\begin{equation}\nonumber
\begin{aligned}
\|(\widetilde B_w)_t\|_{L^2(Q)}
\le  &  C\Big(s\|\rho y\|_{L^2(Q)}+\|\rho y_t\|_{L^2(Q)}+s\|\rho^{-1} w_{tt}\|_{L^2(Q)} +s^2\|\rho^{-1} w_t\|_{L^2(Q)}\\
 &+s^3\|\rho^{-1} w\|_{L^2(Q)}+s\|\rho^{-1} \nabla w_t\|_{L^2(Q)}+s^2\|\rho^{-1} \nabla w\|_{L^2(Q)}\Big)\\
 \le & C\mathcal{A}(u_0,u_1,B)
 \end{aligned}
\end{equation}
and then
\begin{equation}\nonumber 
\|(\rho^{-1}w)_{tt}(t)\|_{L^2(\Omega)} +\|\nabla(\rho^{-1}w)_t(t)\|_{L^2(\Omega)} +\|\Delta( \rho^{-1}w)(t)\|_{L^2(\Omega)}\le C\mathcal{A}(u_0,u_1,B).
\end{equation}
\eqref{estimate_V} then reads  
$\|\rho v\|_{L^\infty(0,T;H^{1/2}(\partial\Omega))}\le C  s\mathcal{A}(u_0,u_1,B)$
which is a part of \eqref{estimate_regular}.

Eventually, from \eqref{linear_control_problem},  we get that
	\begin{equation}\nonumber
	\left\{
		\begin{aligned}
			&(\rho y)_{tt} -\Delta(\rho y) =\widetilde B_y :=\rho B+2\rho_ty_t+\rho_{tt} y-2\nabla\rho\cdot\nabla y-\Delta\rho\, y & \text{in } Q, \\
			&\rho y_{|\Sigma}=\rho v  &  \\
			&((\rho y)(\cdot,0),(\rho y)_t(\cdot,0)) = (\rho(0) u_0,\rho(0)u_1+\rho_t(0)u_0) &\text{in } \Omega,
		\end{aligned}
		\right.
	\end{equation}
for which we have the standard estimate
 $$
\begin{aligned}
\|(\rho y)_t\|_{L^\infty(0,T;L^2(\Omega))}+\|\nabla (\rho y)\|_{L^\infty(0,T;L^2(\Omega))}
\le    &C\Big(\|\widetilde B_y \|_{L^2(Q)}+\|(\rho v)_t\|_{L^2(\Sigma)}+\|\rho v\|_{L^\infty(0,T;H^{1/2}(\partial\Omega))}\\
&+\|\rho(0) u_0\|_{H^1(\Omega)}+\|\rho(0) u_1+\rho_t(0)u_0\|_{L^2(\Omega)}\Big).
\end{aligned}
$$
But
$$
\begin{aligned}
\|\widetilde B_y\|_{L^2(Q)}\le &C\Big(\|\rho B\|_{L^2(Q)} +s^2\|\rho y\|_{L^2(Q)}+s\|\rho y_t\|_{L^2(Q)}+ s\|\nabla(\rho  y) \|_{L^2(Q)} \Big)\\
\le & C s^{1/2} \Big(\mathcal{A}(u_0,u_1,B)+s^{1/2}\| \nabla (\rho  y) \|_{L^2(Q)} \Big),
\end{aligned}
$$
$$
\begin{aligned}
\|\nabla (\rho y)\|_{L^2(Q)} \le  C \Big(\|\rho(0) & u_0\|_{L^2(\Omega)}+\|\rho(0) u_1+\rho_t(0)u_0\|_{L^2(\Omega)}+\| (\rho y)_t\|_{L^2(Q)}+ \|(\rho v)_t\|_{L^2(\Sigma)}\\
&+\|\rho v\|_{L^2(0,T;H^{1/2}(\partial\Omega))}+\|\widetilde B_y\|_{L^2(0,T;H^{-1}(\Omega))}\Big)\end{aligned}
$$
and 
$$
\nonumber
\begin{aligned}
\|\widetilde B_y\|_{L^2(0,T;H^{-1}(\Omega))}
\le & C\Big(\|\rho B\|_{L^2(Q)} +s^2\|\rho y\|_{L^2(Q)}+s\|\rho y_t\|_{L^2(Q)} \Big)\\
\le & Cs^{1/2} \mathcal{A}(u_0,u_1,B)
\end{aligned}
$$
leading to $\|\nabla (\rho y)\|_{L^2(Q)} \le C s^{1/2}\mathcal{A}(u_0,u_1,B)$
that is the estimate of the third term in \eqref{estimate_regular} and 
$\|\widetilde B_y\|_{L^2(Q)} \le Cs^{3/2}\mathcal{A}(u_0,u_1,B)$.
We then deduce that 
$$
\begin{aligned}
\|(\rho y)_t\|_{L^\infty(0,T;L^2(\Omega))}+\|\nabla (\rho y)\|_{L^\infty(0,T;L^2(\Omega))}
\le &
Cs^{3/2}\mathcal{A}(u_0,u_1,B)
\end{aligned}
$$
that is, the estimate of the last two terms in \eqref{estimate_regular}.



\medskip 

\par\noindent
{\bf Step 4.} Case where $B\in L^2(Q)$ and $(u_0,u_1)\in \boldsymbol{V}$. We proceed by density as in \cite[Step 3, page 171]{Bhandari_MCSS_2022}.
}


\bibliographystyle{plain}
\bibliography{ref_wave}

\end{document}